%% file: paper-integrals.tex
\documentclass[margincite, leqno]{sl2art}

\usepackage{mathtools}

\usepackage{tikz}
\usepackage{todonotes}

\makeatletter
\def\qed@warning{}
\makeatother

\numberwithin{equation}{section}

\usepackage[export]{adjustbox}
\usepackage{graphicx,subcaption}
\graphicspath{{fig/}}

\usepackage{enumitem}

\newlist{thmenum}{enumerate}{1}
\setlist[thmenum]{label=(\alph*)}

\input{macros-integrals.tex}

\newkeytheorem{proposition,theorem,lemma,claim,corollary}[style=theorem,parent=section]
\newkeytheorem{definition,remark,example,warning}[style=definition,sibling=proposition]
\newkeytheorem{bigtheorem}[style=theorem,title=Theorem,refname={Theorem,
Theorems}]
\newkeytheorem{conjecture}[style=theorem,title=Conjecture,refname={Conjecture,
Conjectures}]
\newkeytheorem{step}[style=theorem,title=Step,refname={Step,
Steps}]

\usepackage{cochineal}

\usepackage[noabbrev,capitalize]{cleveref}
\crefname{equation}{equation}{equations}
\Crefname{equation}{Equation}{Equations}
\crefname{bigtheorem}{Theorem}{Theorems}
\Crefname{bigtheorem}{Theorem}{Theorems}
\crefname{step}{Step}{Steps}
\Crefname{step}{Step}{Steps}
\crefname{conjecture}{Conjecture}{Conjectures}
\Crefname{conjecture}{Conjecture}{Conjectures}

\title{State integrals for the quantized \(\slg\) Chern-Simons invariant}
\author{Calvin McPhail-Snyder}
\address{Department of Mathematics, Duke University}
\email{calvin@sl2.site}

\begin{document}

\begin{abstract}
  Previous work of the author and N.\ Reshetikhin defines an
  invariant \(\operatorname{Z}_{N}^{\psi}(K, \rho, \mu)\) of a knot
  \(K\), a representation \(\rho \colon \pi_{1}(S^{3} \setminus K) \to
  \operatorname{SL}_{2}(\mathbb{C})\), and a logarithm \(\mu\) of a
  meridian eigenvalue of \(\rho\).
  It can be interpreted as a geometric twist of the Kashaev invariant
  or as a quantization of the \(\operatorname{SL}_{2}(\mathbb{C})\)
  Chern-Simons invariant and is defined via a discrete state-sum
  involving quantum dilogarithms.
  In this paper we show how to express
  \(\operatorname{Z}_{N}^{\psi}(K, \rho, \mu)\) as a sum over contour
  integrals in a space parametrizing hyperbolic structures on the
  knot complement.
  Such state integral presentations are an important step in
  determining the asymptotics of quantum invariants as predicted by
  the Volume Conjecture.
  We discuss this perspective and the remaining obstacles to
  establishing exponential growth of \(\operatorname{Z}_{N}^{\psi}\).
\end{abstract}


\maketitle

\section{Introduction}%
\label{sec:Introduction}

\subsection{Background and motivation}%
\label{sec:Background and motivation}

Kashaev's \defemph{Volume Conjecture} \cite{Kashaev1997} proposes a
striking connection between the geometry of hyperbolic knots and
their quantum invariants.
This conjecture appears to be related to quantum \(\slg\)
Chern-Simons theory, a TQFT arising as a complexification
\cite{Gukov2005} or analytic continuation \cite{Witten2010} of the
more familiar quantum \(\operatorname{SU}(2)\) Chern-Simons theory
\cite{zbMATH04092352}.
A number of authors
\cite{Baseilhac2004,Bonahon2021,Kashaev2005,McPhailSnyderVolume} have
constructed geometric quantum invariants that can be interpreted as
candidates for a mathematical realization of \(\slg\) Chern-Simons theory.

We discuss some typical features of a candidate mathematical realization.%
\note{
  In the terminology of \cref{sec:negative type} these are cyclic
  type invariants.
  There are also related but inequivalent principal type invariants.
}
For concreteness let \(K\) be a hyperbolic knot in \(S^{3}\),
although some versions are additionally or instead defined for closed
\(3\)-manifolds or related objects.
A realization is a family of functions
\[
  \mathcal{Z}_{\nr} \colon \mathsf{X}(K) \to \mathbb{C}
\]
for \(\nr\) a positive integer related to the level and
\(\mathsf{X}(K)\) a space parametrizing hyperbolic structures on
\(\comp{K}\); it is usually a version of the \(\slg\) character
variety of \(\comp{K}\).
Let \(\varsigma_{K} \in \mathbb{C}/ 2 \pi \ii \mathbb{Z}\) be the
\(\slg\) Chern-Simons invariant of the canonical structure
\(\rho_{\hyp}\) of \(K\), normalized so \(2 \pi \Re \varsigma_{K}\)
is the hyperbolic volume of \(K\).
A volume conjecture is the statement that (up to normalization and
some other details)
\begin{equation}
  \label{eq:asymptotics sign positive}
  \lim_{\nr \to \infty}
  \frac{
    \log \mathcal{Z}_{\nr}(K, \rho)
  }{
    \nr
  }
  =
  \varsigma_{K}
\end{equation}
for some \(\rho \in \mathsf{X}(K)\) that is not necessarily \(\rho_{\hyp}\).
We think of \(\nr \to \infty\) as the semiclassical limit.
The surprising observation is that this holds for \(\rho\) other than
\(\rho_{\hyp}\); Kashaev's original conjecture \cite{Kashaev1997} is
that it happens for the trivial hyperbolic structure \(1\).
In this case \(\mathcal{Z}_{\nr}(K,1)\) can be identified with a
value of the \(\nr\)th colored Jones polynomial of \(K\) \cite{Murakami2001}.
There is a standard ansatz (successful in some special cases) for
proving these conjectures that uses two steps.

\begin{step}[Establish integrals]
  \label{step 1}
  Show that the invariant is equivalent (perhaps asymptotically) to a
  \defemph{state integral}, a contour integral
  \[
    \mathcal{Z}_{\nr}(K, \rho) \sim
    \int_{\Gamma} e^{\nr \mathcal{S}(\mathbf{t})} \dif \mathbf{t}
  \]
  where \(\mathcal{S}\) is a holomorphic function that can be
  identified with the Chern-Simons functional.
  Sometimes instead one must consider a sum of such integrals.
\end{step}

In more detail, given a combinatorial description \(\mathcal{T}\) of
\(\comp{K}\) the domain of \(\mathcal{S}\) is a space
\(\Omega_{\mathcal{T}}\) parametrizing hyperbolic structures on the
individual pieces of \(\mathcal{T}\).
Critical points of \(\mathcal{S}\) are choices that combine together
to give a coherent structure on all of \(\comp{K}\) and the critical
value is the \(\slg\) Chern-Simons invariant.
For example, if \(\mathcal{T}\) is an ideal triangulation then
\(\Omega_{\mathcal{T}}\) parametrizes shapes of the tetrahedra and
critical points of \(\mathcal{S}\) are solutions of Thurston's gluing equations.

\begin{step}[Move the contour]
  \label{step 2}
  Show that
  \[
    \int_{\Gamma} e^{\nr \mathcal{S}(\mathbf{t})} \dif \mathbf{t}
    \sim
    \int_{\Gamma'} e^{\nr \mathcal{S}(\mathbf{t})} \dif \mathbf{t}
  \]
  where \(\Gamma'\) is a contour in saddle point position for a
  critical point corresponding to \(\rho_{\hyp}\).
\end{step}

The maximum critical value of \(\Re \mathcal{S}\) is \(\Re \varsigma_{K}\).
It occurs at a saddle point of \(\Re \mathcal{S}\) because
\(\mathcal{S}\) is a holomorphic function.
If \(\Gamma'\) is a contour passing through this saddle point in the
right way a standard \defemph{saddle-point approximation} by a
Gaussian integral gives the desired asymptotics \eqref{eq:asymptotics
sign positive}.

In general both steps require hard computations; a detailed
discussion of the analytic issues is given by
\textcite{zbMATH06684922}.
\cref{step 1} requires re-writing a discrete sum as a sum of
integrals, typically by using the Poisson summation formula.
Ensuring that this is valid requires careful estimates on the size of
\(\mathcal{S}\) and related functions.

Previous work of the author \cite{McPhailSnyderVolume} uses the
holonomy \(R\)-matrices defined in joint work with N.\ Reshetikhin
\cite{McPhailSnyderAlgebra} to define a candidate mathematical
realization \(\qinv{}\) of complex Chern-Simons theory as described above.
This construction is based on work of \textcite{Kashaev2005}.
\(\qinv{}\) is closely related to the invariants studied by
\citeauthor{Baseilhac2004}
\cite{Baseilhac2004,Baseilhac2007,Baseilhac2011kashaev} and can be understood
as a lifting or refinement of them.

In this paper we show that \(\qinv{}\) is equal to a sum of state integrals.
Our method does not require checking any inequalities and works for
any representation satisfying a natural, algebraic nondegeneracy condition.
To the author's knowledge this is the first example of an exact (not
asymptotic) state integral presentation of a cyclic type invariant.
Such integral presentations automatically occur for principal type
invariants, an inequivalent but related mathematical realization of
complex Chern-Simons theory (\cref{sec:negative type}).
Our results may clarify the relationship between these two branches.

\subsection{Our results}%
\label{sec:Our results}

Let \(K\) be an oriented, framed knot in \(S^{3}\) and \(\rho\) a
decorated representation of \(K\), which is a representation \(\rho \colon
\pi_{1}(\comp{K}) \to \slg\) along with a choice of invariant line for
the image of the peripheral subgroup under \(\rho\).
In particular this choice distinguishes an eigenvalue \(m\) of the
meridians of \(K\).
A \defemph{log-meridian} is a complex number \(\mu\) with \(e^{\tu \mu} = m\).
The main result of \cite{McPhailSnyderVolume} is the definition of a
sequence of invariants
\[
  \qinv{K, \rho, \mu} \in \mathbb{C}
\]
for integers \(\nr \ge 2\).
We can view them as a quantization of the \(\slg\) Chern-Simons
invariant, as a geometric twist of the Kashaev \cite{Kashaev1995} and
ADO \cite{Akutsu1992} invariants, or as a candidate mathematical
realization of \(\slg\) Chern-Simons theory.
They are related to the invariants of \textcite{Baseilhac2004}; see
\cite[Section 1.6.1]{McPhailSnyderVolume} for a more detailed discussion.
The construction makes sense (and our results do too) for links and
tangles as well, but for simplicity we restrict to the case of knots.

Let \(D\) be an open knot diagram of \(K\) (i.e.\ a \(1\)-\(1\)
tangle diagram whose closure is \(K\)).
Let \(E\) be the set of internal segments (edges) of \(D\).
Using quantum dilogarithms we define an \defemph{action} \(\actq{D,
\mu}{} \colon \mathbb{C}^{E} \to \mathbb{C}\) so that
\[
  t \mapsto e^{ \nr \actq{D, \mu}{t} } \colon  \mathbb{C}^{E} \to \mathbb{C}
\]
is a meromorphic function.
Given the choice of some additional gauge data \(u\) called a
\defemph{shadow coloring} the invariant of \cite{McPhailSnyderVolume}
is by definition a state sum
\[
  \qinv{K, \rho, \mu}
  =
  \frac{1}{\nr^{\operatorname{cr}(D)}}
  \sum_{n \in \nset{\nr}^{E}}
  e^{ \nr  \actq{D, \mu}{(\beta + n)/\nr} }
\]
where \(\nset{\nr} = \set{0, 1, \dots, \nr - 1}\) and
\(\operatorname{cr}(D)\) is the crossing number of \(D\).
The vector \(\beta \in \mathbb{C}^{E}\) encodes \((\rho, u)\).
For a vector \(k \in \mathbb{Z}^{E}\) with integer entries%
\note{%
  As discussed in \cref{rem:flattening vector} the choice of
  \(\mathbf{k}\) (along with \(\beta\)) determines a flattening
  of the octahedral decomposition associated to \(D\).
}
a \defemph{state integral} is
\begin{equation}
  \label{eq:state integral intro}
  \stint{D, \mu, \beta}{k}
  \defeq
  \int_{[0,1]^{E}}
  \exp \leftfun[ \nr ( \actq{D, \mu}{\beta/\nr + t} - \tu k \cdot t )
  \rightfun]
  \dif t
\end{equation}
In \cref{sec:State integrals big} we prove (\cref{main result}) that
\(\qinv{}\) is given by the symmetric rectangular sum of all state integrals:
\begin{equation}
  \label{eq:state integral sum intro}
  \qinv{K, \rho, \mu}
  =
  \nr^{(|E| - 1)/2}
  \lim_{\min_{i\in E}M_i\to\infty}
  \sum_{\substack{k\in\mathbb Z^E\\ |k_i|\leq M_i\ (i\in E)}}
  \stint{D, \mu, \beta}{k}.
\end{equation}
The limit above is said to converge \defemph{in the Pringsheim
sense}, which is a form of conditional convergence.
We discuss why we do not expect absolute convergence of
\eqref{eq:state integral sum intro} in \cref{rem:absolute convergence}.

To the author's knowledge \cref{main result} is the first exact
description of a cyclic type invariant in terms of state integrals.
It is inspired by similar state integrals for principal type
invariants (discussed below) and by earlier approximate expressions
in terms of state integrals. The standard strategy (for example in
\cite{zbMATH06684922}) for establishing large \(\nr\) asymptotics of cyclic
invariants is to show that the invariant in question is
\emph{asymptotically} given by an integral or a sum of such integrals
as \(\nr \to \infty\). Here by asymptotically we mean up to an error
term decaying sufficiently rapidly in \(\nr\). While our sum is
infinite it is exact for each \(\nr\) and the infinite sum is natural
from the point of view of hyperbolic geometry (\cref{rem:flattening
vector}).

Previous constructions of cyclic type invariants have phase
ambiguities: they are only defined up to an overall root of unity.
One reason we describe \(\qinv{}\) as a refinement of earlier work
is that (given the choice of log-meridian \(\mu\)) it has a well-defined phase.
A related fact is that the exponentiated action \(e^{\nr \actq{D,
\mu}{}}\) does not have any branch cuts;
the analogue of this function in (for example) the work of
\textcite{Baseilhac2004} is multivalued. A single-valued
integrand greatly simplifies the proof of \cref{main result}.

The other technical issue in the proof of \cref{main result} is
ensuring the contour of integration
avoids the poles of \(e^{\nr\actq{D, \mu}{}}\).
We discuss the required algebraic condition in \cref{sec:Gauge
transformations and admissibility}.
Once this is handled using the methods of
\cite{McphailSnyder2024octahedralcoordinateswirtingerpresentation}
rewriting the sum in terms of integrals is a straightforward
application of standard results on multiple Fourier series.

Because \(\qinv{}\) is closely related to the Baseilhac-Benedetti
invariant \cite[Section 1.6.1]{McPhailSnyderVolume} and thus to the
Bonahon-Liu-Wong-Yang invariants \cite{arXiv:2601.03554} our result
may show how to write these invariants in terms of state integrals.
The key ingredients to prove \cref{main result} are
\begin{enumerate}
  \item a well-behaved integrand without branch cuts, and
  \item avoiding singularities in the geometric parameter space.
\end{enumerate}
The first can be obtained by using Faddeev's quantum dilogarithm instead
of other versions (see \cite[Appendix A]{McPhailSnyderVolume} for a
longer discussion).
The second is harder to ensure in general: for invariants defined in terms
of a triangulation it is related to the values of the shape
parameters \(z_{i}\) of the tetrahedra.
Our method needs the \(z_{i}\) to avoid not only the singular values
\(1, 0, \infty\) but also the unit circle.
Even for geometric triangulations this does not always hold: for
example, the standard Thurston triangulation of the figure-eight
complement has both shape parameters equal to \(e^{\pi \ii/3}\), so
naively following our integral prescription will give contours
colliding with the poles of the quantum dilogarithm.
The construction of \(\qinv{}\) uses semi-ideal octahedral
decompositions that are much more complicated than the minimal
geometric triangulations;
in \cref{ex:figure eight} the triangulation of the figure-eight
complement has \(16\) tetrahedra. The benefit is gaining enough
degrees of freedom to avoid the unit circle as required.

We explain why \eqref{eq:state integral sum intro} is useful from the
point of asymptotics in \cref{sec:Formal asymptotics}.
The idea is that, as \(\nr \to \infty\), the functions \(\actq{D ,
\mu}{}\) converge to a function \(\act{D}{}\) that we interpret as
the Chern-Simons action. In the literature \(\act{D}{}\) is
frequently called the \defemph{(Yokota) potential function}
\cite{arXiv:math/0009165} and denoted \(V\). The critical points of
\(\act{D}{}\) are boundary-parabolic representations of \(K\) and
the critical values are their Chern-Simons invariants. As such
\cref{main result} accomplishes \cref{step 1} towards proving a
volume conjecture for \(\qinv{}\). We state this precisely, explain
how the sum over \(\mathbb{Z}^{E}\) is natural, and discuss barriers
to \cref{step 2} in \cref{sec:Formal asymptotics}. It seems that new
ideas are required to complete \cref{step 2} in general.

\subsection{Cyclic and principal type invariants}
\label{sec:negative type}

There appear to be two inequivalent (but perhaps both natural)
branches of complex Chern-Simons theory.
Above we described \defemph{cyclic type} invariants
\cite{Baseilhac2004,Bonahon2021,Kashaev2005,McPhailSnyderVolume} that
\begin{itemize}
  \item are defined as a sum over a finite state space,
  \item are conjectured to exponentially grow as \(\nr \to \infty\), and
  \item are well-defined for every (or every geometrically
    non-degenerate) hyperbolic structure.
\end{itemize}
In contrast, \defemph{principal type} invariants
\cite{zbMATH06324352,arXiv:1305.4291,arXiv:2508.05120}
\begin{itemize}
  \item are defined using integrals over a noncompact, continuous state space,
  \item are conjectured to \emph{decay} as \(\nr \to \infty\), but
  \item are only defined when the hyperbolic structure satisfies a
    positivity condition so that the integrals converge.
\end{itemize}
These two branches are related to different quantizations of the moduli space of
\(\slg\) local systems.
Principal type invariants are related to the representations of
quantum cluster varieties on function spaces studied by
\textcite{zbMATH01154746} and \textcite{zbMATH05530519}.
These are infinite-dimensional representations analogous to the
principal series representations of \(\operatorname{SL}_{2}\).
The cyclic representations of \cite[Section 2.1.2]{zbMATH01154746}
further studied by \textcite{zbMATH05220932}
(see \textcite{arXiv:2501.02316} for a
systematic overview) are finite-dimensional and analogous to
highest-weight modules.
Cyclic type invariants are more directly related to the Kashaev
invariant and the original Volume Conjecture.
It is typically easier to prove volume conjectures for principal type
invariants:
they are \emph{defined} in terms of integrals, so \cref{step 1} holds
automatically, subject to the positivity condition.

Asymptotic growth of cyclic type invariants has been established
only for knots with up to \(7\) crossings \cite{zbMATH06684922,
Ohtsuki2017, Ohtsuki2018} and for twist knots \cite{Chen2023II}.
For principal type invariants much more is known:
the invariant of \textcite{arXiv:2508.05120} satisfies the conjecture
for every hyperbolic manifold with totally geodesic boundary,
and the Andersen-Kashaev Teichm{\"u}ller TQFT satisfies it for twist
knots \cite{zbMATH07761477} and many other manifolds
\cite{arXiv:2410.10776, arXiv:2512.17437, arXiv:2512.23198}.
We hope that our result leads to a proof of asymptotic growth of
\(\qinv{}\) using similar techniques.

In addition, our work may lead to a better understanding of the
relationship between cyclic and principal invariants.
\textcite{arXiv:1305.4291} showed that their principal type TQFT can
be defined using integrals over compact state spaces.
This form of the partition function formally resembles the sum of
state integrals \eqref{eq:state integral sum intro} after a Wick
rotation (i.e.\ rotating the integration contour by \(\pi/2\) in the
complex plane) and changing the overall sign of the action.
Their reformulation is accomplished via a Weil-Gel'fand-Zak transform
taking values in a space of quasi-periodic functions.
Similar quasi-periodic functions are used in the definition of
\(\qinv{}\) \cite{McPhailSnyderVolume} and this may lead to a more
precise relationship between the two constructions.

\subsection*{Acknowledgements}
I would like to thank Francis Bonahon, Effie Kalfagianni, Lillian
Pierce, and Ka Ho Wong for helpful conversations.
I would also like to thank an anonymous referee for suggestions that
significantly improved the paper, in particular
\cref{sec:Preliminaries,sec:Asymptotics} and \cref{main result}.
The revised proof of \cref{main result} was developed via an AI agent
and subsequently re-written by myself.
I take full responsibility for its correctness.

\section{Preliminaries}%
\label{sec:Preliminaries}

The domains of our state integrals lie in a space parametrizing hyperbolic
structures on the knot exterior. Coordinates on this space depend on
the choice of diagram via its associated octahedral decomposition. As
such, to prove our main result we first need to give some
facts about \(\slg\) representations of link complements
(i.e.\ hyperbolic structures) and the associated coordinates. These
are discussed in detail in
\cite{McphailSnyder2022hyperbolicstructureslinkcomplements} and
\cite{McphailSnyder2024octahedralcoordinateswirtingerpresentation}.

\subsection{Tangle diagrams and colorings}%
\label{sub:Diagrams}

We choose slightly nonstandard conventions that are convenient for
the results of
\cite{McphailSnyder2022hyperbolicstructureslinkcomplements,McphailSnyder2024octahedralcoordinateswirtingerpresentation}.
We use the topological convention for \(\pi_{1}\): its product is
concatenation of paths read left-to-right.
A representation into \(\slg\) is a homomorphism \(\pi_{1} \to \slg\).
Thinking of \(\slg\) as a set of matrices it has a natural right action
on row vectors in \(\mathbb{C}^{2}\).

\begin{definition}
  \label{def:decorated representation}
  A \defemph{decorated representation} of a knot \(K\) is a
  representation \(\rho \colon \pi_{1}(\comp{K}) \to \slg\) and a choice
  of invariant line for the image of each peripheral subgroup of \(K\)
  under \(\rho\).
\end{definition}

More formally, given a choice of peripheral subgroup \(\Pi \subset
\pi_{1}(\comp{K})\), a decoration is a one-dimensional subspace \(L
\subset \mathbb{C}^{2}\)
invariant under this right action of \(\rho(\Pi)\), i.e.\ an eigenline.
This does not depend on the choice of peripheral subgroup: any other
choice is a conjugate \(y^{-1} \Pi y\) and the corresponding
invariant line is \(L \rho(y)\).
An orientation of \(K\) determines a meridian \(\mer \in \Pi\) and
the decoration \(L\) determines an eigenvalue \(m \in \mathbb{C}^{\times}\) by
\[
  v \rho(\mer) = m^{-1} v
\]
where \(v\) is any nonzero element of \(L\).
The inverse is unfortunate but correct and is chosen to match the
conventions of \cite{McPhailSnyderAlgebra,McPhailSnyderVolume}.
Clearly \(m\) does not depend on the choice of \(\Pi\).

\begin{definition}
  \label{def:open diagram}
  Let \(K\) be an oriented, framed%
  \note{
    Recall that a \defemph{framing} of a knot is a choice of non-vanishing
    normal vector field up to isotopy.
  }
  knot in \(S^{3}\).
  An \defemph{open diagram} of \(K\) is an oriented, framed
  \((1,1)\)-tangle\note{
    Recall that a \defemph{tangle} is (the isotopy class rel boundary of) an
    embedding of circles and arcs into \([0,1] \times
    \mathbb{R}^{2}\) so that the ends of the arcs lie on the boundary
    of \([0,1] \times \mathbb{R}^{2}\).
    An oriented \((1,1)\)-tangle is one with a single incoming and
    outgoing boundary component;
    it is sometimes called an ``open knot'' or ``string knot''.
  }
  diagram \(D\) whose closure is \(K\).
  We take \(D\) to have the blackboard framing.%
  \note{
    That is, the framing chosen by pushing off from the knot parallel
    to the plane in which the diagram is drawn.
  }
  Our convention is to draw \(D\) with its incoming end on the left
  and its outgoing end on the right, as in \cref{ex:figure eight}.

  The \defemph{segments} of \(D\) are the edges of its underlying graph.
  We call the two segments incident to the boundary of the tangle the
  \defemph{boundary segments}, call the remainder the
  \defemph{internal segments}, and write \(E\) for the set of
  internal segments of \(D\).
  Writing \(\operatorname{cr}(D)\) for the crossing number of the knot
  diagram we have \(2 \operatorname{cr}(D) = |E| + 1\).
\end{definition}

\begin{definition}
  \label{def:decorated matrix}
  A \defemph{decorated matrix} is a pair \((g, [v]) \in \slg \times
  \pjsp\) of a matrix \(g\) and a \(g\)-invariant eigenline \([v]\),
  where \(v\) is a nonzero row vector.
  The \defemph{distinguished eigenvalue} of \((g, [v])\) is the \(m
  \in \mathbb{C}^{\times}\) satisfying
  \[
    vg = m^{-1} v.
  \]
  We define a binary operation on decorated matrices by
  \begin{equation}
    (g, [v]) \qn (h, [w])
    \defeq
    (h^{-1} g h, [v h]).
  \end{equation}
  A \defemph{decorated \(\slg\)-coloring}%
  \note{With the operation \(\qn\) the decorated matrices form a \defemph{quandle} \cite{zbMATH00165698}; the coloring defined below is the standard quandle coloring of a tangle diagram.}
  of a tangle diagram \(D\)
  assigns decorated matrices to every arc of \(D\) subject to the
  crossing relations
  \begin{equation}
    \begin{aligned}
      &
      \begin{tikzpicture}[line width=1, baseline=30, scale=1]
        \draw[->] (0,0) node[left] {\(g_{2}\)} \br (1.5,1)
        node[right] {\(g_{2} \qn g_{1}\)};
        \draw[white, line width=10] (0,1) node[left] {} \br (1.5,0)
        node[right] {};
        \draw[->] (0,1) node[left] {\(g_{1}\)} \br (1.5,0)
        node[right] {\(g_{1}\)};
      \end{tikzpicture}
      &
      &
      \begin{tikzpicture}[line width=1, baseline=30, scale=1]
        \draw[->] (0,1) node[left] {\(g_{1}\)} \br (1.5,0)
        node[right] {\(g_{1} \qninv g_{2}\)};
        \draw[white, line width=10] (0,0) node[left] {} \br (1.5,1)
        node[right] {};
        \draw[->] (0,0) node[left] {\(g_{2}\)} \br (1.5,1)
        node[right] {\(g_{2}\)};
      \end{tikzpicture}
    \end{aligned}
  \end{equation}
  where \(\qninv\) is the inverse of \(\qn\) defined by
  \(
    (g, [v]) \qninv (h, [w])
    \defeq
    (h g h^{-1}, [v h^{-1}]).
  \)
  We always require that a \(\slg\)-coloring of an open diagram
  assigns the same decorated matrix to both boundary segments, or
  equivalently that it gives a well-defined coloring of the closure.
\end{definition}

The Wirtinger presentation of \(\pi_{1}(\comp{K})\) makes it clear
that decorated representations of \(K\) are in bijection with
decorated \(\slg\)-colorings of \(D\) so we denote both by \(\rho\).
To define our invariant we need to make an additional choice of
diagram-dependent gauge data.%
\note{
  Our invariant is defined in terms of the octahedral decomposition
  associated to \(D\).
  This is a semi-ideal triangulation with two non-ideal vertices.
  The extra degrees of freedom are related to the choice of location
  of these vertices.
}
The value of \(\qinv{}\) is independent of this choice.

\begin{definition}
  \label{def:shadow coloring}
  A \defemph{shadow coloring} \((\rho, u)\) of a diagram \(D\) is a
  decorated \(\slg\)-coloring \(\rho\) of \(D\) along with an
  assignment \(u \colon j \mapsto u_{j}\) of column vectors \(u_{j} \in
  \mathbb{C}^{2}\) to the regions \(j\) of \(D\), subject to the rule
  \begin{equation}
    \begin{tikzpicture}[line width=1, baseline=00, scale=1]
      \draw[->] (0,0) node[left] {\((g, [v])\)} to (1.5,0) node[right] {};
      \node[above] at (.75,0.25){\(u\)};
      \node[below] at (.75,-0.25){\(g u\)};
    \end{tikzpicture}
  \end{equation}
  In the diagram above, the region labeled by \(u\) lies on the
  \defemph{left} of the oriented segment.
\end{definition}

It is not hard to see that a shadow coloring is determined uniquely
by the vector \(u_{0} \in \mathbb{C}^{2}\) assigned to the region above all
strands in the chosen left-to-right tangle presentation (or to any other
region).
A shadow coloring of \(D\) fixes the geometric data of the associated
octahedral decomposition in terms of certain parameters defined below.

\begin{definition}
  \label{def:parameters}
  A shadow coloring of a diagram \(D\) assigns complex numbers to the
  regions and segments of \(D\).
  Write \(\evec{1}=(1,0)\)  and \(\evec{2} = (0,1)\) for the standard
  basis row vectors.
  We assign \(D\) the \defemph{region parameters}
  \begin{equation}
    \label{eq:region parameters}
    a_{j} \defeq \evec{1} u_{j}
    \text{ for each region \(j\)}
  \end{equation}
  and \defemph{segment parameters}
  \begin{equation}
    \label{eq:segment parameters}
    b_{i} \defeq
    -
    \frac{
      v_{i} \evec{2}^{\operatorname{T}}
    }{
      v_{i} u_{i^{\operatorname{L}}}
    }
    \text{ for each segment \(i\)}
  \end{equation}
  where \(i^{\operatorname{L}}\) is the region on the left of the
  oriented segment \(i\):
  \[
    \begin{tikzpicture}[line width=1, baseline=00, scale=1]
      \draw[->] (0,0) node[left] {\((g_{i}, [v_{i}])\)} to (1.5,0)
      node[right] {};
      \node[above] at (.75,0.25){\(u_{i^{\operatorname{L}}}\)};
    \end{tikzpicture}
  \]
  Here the vectors \(u_j\) are columns and the representatives
  \(v_i\) of eigenlines are rows, so the products above are the
  natural pairing of row and column vectors:
  \[
    \evec{1}
    \begin{bmatrix}
      u^{0} \\ u^{1}
    \end{bmatrix}
    =
    \begin{bmatrix}
      1 & 0
    \end{bmatrix}
    \begin{bmatrix}
      u^{0} \\ u^{1}
    \end{bmatrix}
    =
    u^{0}
    \text{ and }
    \begin{bmatrix}
      v^{0} & v^{1}
    \end{bmatrix}
    \evec{2}^{\operatorname{T}}
    =
    v^{1}.
  \]
  If none of the region or segment parameters are \(0\) or \(\infty\)
  we say the shadow coloring is \defemph{admissible}.
\end{definition}

\begin{definition}
  A shadow coloring \((\rho, u)\) of an open diagram \(D\) is
  \defemph{normalized} if it is admissible and the segment parameter
  assigned to the open segments is \(1\).
\end{definition}

Clearly one can always choose a normalized coloring by rescaling the
vectors \(u_{i}\).
This is a special case of the type (B) gauge transformation defined below.

\subsection{Gauge transformations and nondegeneracy }%
\label{sec:Gauge transformations and admissibility}

The diagram \(D\) determines a semi-ideal triangulation
\(\mathcal{T}_{D}\) of \(\comp{K}\) (with two non-ideal points)
called the \defemph{octahedral decomposition};
we refer to \cite{Kim2018} for a detailed discussion and
\cite{McphailSnyder2022hyperbolicstructureslinkcomplements,
McphailSnyder2024octahedralcoordinateswirtingerpresentation} for the
relationship to our parameters.
\(\mathcal{T}_{D}\) assigns each crossing of \(D\) four tetrahedra,
with one in each corner of the crossing.
The segment parameters naturally parametrize hyperbolic structures on
these tetrahedra in terms of \defemph{shape parameters} \(z_{k}\).
Labeling the segments and corners of a crossing as
\begin{equation}
  \label{eq:crossing labeled}
  \begin{tikzpicture}[line width=1, baseline=30, scale=1]
    \draw[->] (0,0) node[left] {\({2}\)} \br (1.5,1) node[right] {\({2'}\)};
    \draw[->] (0,1) node[left] {\({1}\)} \br (1.5,0) node[right] {\({1'}\)};
    \draw (0.75,1) node {\(\lN\)};
    \draw (0,0.5) node {\(\lW\)};
    \draw (0.75,0) node {\(\lS\)};
    \draw (1.5,0.5) node {\(\lE\)};
  \end{tikzpicture}
\end{equation}
the shape parameters of the tetrahedra are given by
\begin{equation}
  \label{eq:shape parameters}
  \begin{aligned}
    z_{\operatorname{N}} &= \left( \frac{b_{2'}}{b_{1}} \right)^{\epsilon}
    &
    z_{\operatorname{S}} &= \left( \frac{b_{2}}{b_{1'}}\right)^{\epsilon}
    \\
    z_{\operatorname{W}} &= \left(\frac{b_{2}}{m b_{1}}\right)^{\epsilon}
    &
    z_{\operatorname{E}} &= \left(\frac{m b_{2'}}{b_{1'}}\right)^{\epsilon}
  \end{aligned}
\end{equation}
where \(\epsilon\) is the sign of the crossing, the \(b_{i}\) are the
segment parameters of \cref{eq:segment parameters}, and \(m\) is the
distinguished eigenvalue of any segment because we are considering
knots, which have a single peripheral subgroup.\note{For multiple
  components one considers eigenvalues \(m_{1}, m_{2}\) as in
\cite[Definition 2.12]{McPhailSnyderVolume}.}
The shape parameters arise as cross ratios of the positions of the
vertices in \(\pjsp\).
In particular a tetrahedron is geometrically degenerate when its
shape parameter lies in \(\set{0, 1, \infty}\).

If a shadow coloring is admissible as defined above then none of the
shape parameters of \(\mathcal{T}_{D}\) are \(0\) or \(\infty\).%
\note{
  Strictly speaking a shape parameter \(z = z^{0}\) corresponds to a
  pair of edges of a vertex-ordered tetrahedron, with the other edge
  pairs assigned \(z^{1} = 1/(1 - z^{0})\) and \(z^{2} = 1 - 1/z^{0}\).
  Admissibility means that a particular preferred edge pair is not
  \(0\) or \(\infty\).
}
This is a necessary condition to define \(\qinv{}\), and every
representation \(\rho\) is gauge-equivalent to an admissible one.
Since \(\qinv{}\) is gauge-invariant this allows us to define it for
any \(\rho\).

\begin{definition}[\protect{\cite[Section
  4]{McphailSnyder2024octahedralcoordinateswirtingerpresentation}}]
  \label{def:gauge transformation}
  Let \((\rho, u)\) be a shadow coloring of \(D\) with segment colors
  \(i \mapsto (g_{i}, [v_{i}])\) and region colors \(j \mapsto u_{j}\).
  A \defemph{gauge transformation} of it is
  \begin{itemize}
    \item[(A)]  a coloring of the form \(i \mapsto (h^{-1} g_i h, L_i
      h)\) and \(j \mapsto
      h^{-1} u_j\) for some \(h \in \slg\), or
    \item[(B)] a coloring obtained by preserving all the segment colors and
      replacing one region color \(u_{0}\) with another nonzero
      column vector \(u_{0}'\), then replacing the remaining region colors
      with the unique choices consistent with \(u_{0}'\).
  \end{itemize}
  We say the new colorings are the images of \defemph{type (A)} and
  \defemph{type (B)} gauge transformations  \defemph{by \(h\)}.
\end{definition}

\begin{remark}
  The underlying definition of a type (B) transformation is diagrammatic,
  as in \cite[Figure
  8]{McphailSnyder2024octahedralcoordinateswirtingerpresentation}.
  The characterization of type (B) transformations given here
  correctly matches this underlying definition.
  However, it does \emph{not} agree with the definition given in
  \cite[Definition
  4.1]{McphailSnyder2024octahedralcoordinateswirtingerpresentation},
  which is incorrect.
  The author thanks an anonymous referee for noticing this error.

  Because the proof of \cite[Theorem
  2]{McphailSnyder2024octahedralcoordinateswirtingerpresentation}
  uses the wrong characterization of type (B) transformations it is
  incorrect as written.
  However, the underlying idea (due to \textcite{Cho2014}) is valid,
  and the conclusion of the theorem is still true.
  Below we give a corrected proof.
\end{remark}

\begin{lemma}[\protect{\cite[Theorem
  2]{McphailSnyder2024octahedralcoordinateswirtingerpresentation}}]
  \label{thm:shadow coloring gauge admissible}
  Let \(\rho\) be a decorated representation and \(D\) an open diagram of \(K\).
  Any shadow coloring \((\rho, u)\) of \(D\) is gauge-equivalent to
  one with an admissible shadow coloring.
\end{lemma}

\begin{proof}
  Write \([v_{i}]\) for the eigenspaces associated to the segments by
  \((\rho, u)\) and \(u_{j}\) for the column vectors associated to the regions.
  By \cref{def:parameters} \((\rho, u)\) is admissible if and only if
  the quantities
  \begin{equation*}
    v_{i} \evec{2}^{\operatorname{T}}, \evec{1} u_{j},  \text{ and }
    v_{i} u_{i^{\operatorname{L}}}
  \end{equation*}
  are nonzero for every segment \(i\) and region \(j\).
  (Here \(i^{\operatorname{L}}\) is the region lying to the left of
  segment \(i\).)

  We can ensure the first condition by applying a type (A) transformation.
  There are finitely many \(v_{i}\) and they transform as \(v_{i}
  \mapsto v_{i}h\) under a type (A) transformation by \(h\).
  Call \(h \in \slg\) \defemph{bad} if after a type (A)
  transformation by \(h\), we still have \(v_{i} h
  \evec{2}^{\operatorname{T}} = 0\).
  Thus the set of bad \(h \in \slg\) is the union of the nonempty
  Zariski closed sets
  \begin{equation*}
    X_{i}
    =
    \set{h \in \slg \given v_{i} h \evec{2}^{\operatorname{T}} = 0}.
  \end{equation*}
  Since no \(X_{i}\) is all of \(\slg\) their union \(X = \cup_{i}
  X_{i}\) is not either.
  We may therefore apply a type (A) transformation by any \(h \in
  \slg \setminus X\) to ensure that \(v_{i}
  \evec{2}^{\operatorname{T}} \ne 0\) for each \(i\).
  This condition does not involve the \(u_{j}\) so is invariant under
  type (B) transformations.

  We have reduced the problem to showing that, after a type (B) gauge
  transformation,
  \begin{equation*}
    \evec{1} u_{j}, v_{i} u_{i^{\operatorname{L}}} \ne 0
    \text{ for every segment } i \text{ and region } j.
  \end{equation*}
  Writing \(u_{0}\) for the vector associated to the topmost region
  of \(D\), for each region \(j\)
  \begin{equation*}
    u_{j} = G_{j} u_{0}
  \end{equation*}
  for some product \(G_{j} = g_{i_{1}}, \dots g_{i_{N}}\) of the
  values of \(\rho\) on the Wirtinger generators.
  We think of these as fixed, as they do not change under type (B)
  transformations, so now our goal is to show that
  \begin{equation*}
    \evec{1} G_{j} u_{0},  v_{i} G_{i^{\operatorname{L}}} u_{0} \ne 0
  \end{equation*}
  for every segment \(i\) and region \(j\).
  Type (B) transformations let us change \(u_{0}\) to any nonzero
  vector in \(\mathbb{C}^{2}\), so a similar Zariski density argument
  (there are finitely many regions and segments) shows the set of
  admissible \(u_{0}\) is nonempty.
\end{proof}

While we can ensure the shape parameters lie in \(\mathbb{C}\setminus
\set{0}\) we cannot necessarily ensure they avoid \(1\).
Having shape parameter \(1\) is independent of the choice of shadow
coloring and is equivalent to a natural condition on the decorated
representation:

\begin{definition}
  A crossing with a decorated \(\slg\)-coloring
  \[
    \begin{tikzpicture}[line width=1, baseline=30, scale=1]
      \draw[->] (0,0) node[left] {\((g_{2}, [v_{2}])\)} \br (1.5,1)
      node[right] {\((g_{2'}, [v_{2'}])\)};
      \draw[->] (0,1) node[left] {\((g_{1}, [v_{1}])\)} \br (1.5,0)
      node[right] {\((g_{1'}, [v_{1'}])\)};
    \end{tikzpicture}
  \]
  is \defemph{pinched} if \([v_{1}] = [v_{2}]\), equivalently
  \([v_{i}] = [v_{j}]\) for any pair \((i,j)\) of adjacent segments.
  (Note that this condition does not depend on a choice of shadow coloring.)
  A coloring of a diagram is pinched if any of its crossings are.
  If \(\rho\) is a decorated representation of \(K\) and \(D\) an
  open diagram of \(K\) we say \(D\) is  \defemph{\(\rho\)-pinched}
  if the coloring induced by \(\rho\) is pinched (and conversely that
  \(\rho\) is  \defemph{\(D\)-pinched}).
  If no crossing is pinched we say \((D, \rho)\) is \defemph{smooth}.
\end{definition}

A nugatory crossing is always pinched.
If \(D\) is a crossing-minimal diagram it has no nugatory crossings
and in this case most (but not all) nontrivial representations
\(\rho\) are \(D\)-smooth.
For example, up to gauge transformation the \(8_{5}\) knot has \(11\)
boundary-parabolic representations \cite{Kim2018}.
In the standard diagram \(D\) \cite[Figure 28]{Kim2018} of this knot
\(10\) are smooth and \(1\) is pinched.
\(8_{5}\) is hyperbolic and its discrete faithful representation is
\(D\)-smooth.
According to the database of boundary-parabolic representations computed by
\textcite{DiagramSite} this pattern holds in general:
for crossing-minimal \(D\) most but not all representations are
\(D\)-smooth, but the hyperbolic holonomy representation is always smooth.

In the standard terminology, a diagram that is not \(\rho_{\hyp}\)-pinched is
said to be \defemph{essential}.%
\note{
  More formally the associated octahedral decomposition is essential.
}
By \cite[Theorem
3]{McphailSnyder2024octahedralcoordinateswirtingerpresentation} being
essential is equivalent to a purely group-theoretic condition called
\defemph{arc-faithfulness}.
It is not hard to see that if \(D\) is not arc-faithful then it is
pinched for every decorated representation \(\rho\).
It is natural to conjecture (see \cite[Section
5]{McphailSnyder2024octahedralcoordinateswirtingerpresentation}) that
any hyperbolic knot has an arc-faithful diagram.
It is known \cite{Garoufalidis2016,Sakuma2018} that any reduced
alternating diagram of a hyperbolic knot
is arc-faithful; see
\cite[Corollary
5.9]{McphailSnyder2024octahedralcoordinateswirtingerpresentation}.

At any pinched crossing the shape parameters are \(1\) so the
\(R\)-matrix elements defining \(\qinv{}\) encounter the poles and
zeros of the quantum dilogarithm.
One can show that these cancel, but in a delicate way that prevents
our state integrals from being well-defined.
More generally to avoid the singularities we need to ensure that no
shape parameter has norm \(1\), and if no crossings are pinched then
we can achieve this via a gauge transformation.

\begin{lemma}[\protect{\cite[Corollary
  5.3]{McphailSnyder2024octahedralcoordinateswirtingerpresentation}}]
  \label{thm:avoid unit circle}
  Let \(\rho\) be a decorated representation of a knot \(K\) and
  \(D\) be an open diagram of \(K\).
  Suppose that \(D\) is not \(\rho\)-pinched.
  Then there is a representation \(\rho'\) gauge-equivalent to
  \(\rho\) and a shadow coloring \((\rho', u)\) of \(D\) in which  no
  shape parameter has norm \(1\).
  In this case we say \((\rho', u)\) \defemph{avoids the unit circle}.
\end{lemma}

Unlike for \cref{thm:shadow coloring gauge admissible} the proof in
\cite{McphailSnyder2024octahedralcoordinateswirtingerpresentation} is
correct as written.
The idea is to show that the set of bad type (A) transformations
(that leave some shape parameter lying on the unit circle) is parametrized
by a finite union of circles in the Riemann sphere.
Since these cannot be the entire sphere a good type (A) transformation exists.

\subsection{Segment equations}%
\label{sec:Segment equations}

The segment parameters have a natural geometric
interpretation: they parametrize the locations of the ideal vertices
of the octahedral decomposition.
Following
\cite{McphailSnyder2024octahedralcoordinateswirtingerpresentation} we
view the segment parameters as coming from a shadow coloring,
but one can also work in the opposite direction.
An assignment of numbers \(b_{i}\) to the segments of \(D\) induces a
decorated representation of \(K\) if and only if the shape parameters
satisfy Thurston's gluing equations.
These can be summarized in one equation for each segment.

\begin{definition}
  \label{def:segment equations}
  Let \(D\) be an open diagram of a knot \(K\).
  Choose a meridian eigenvalue \(m \in \mathbb{C} \setminus \set{0}\)
  (as in \cref{def:decorated matrix}) for one, hence every meridian of \(K\).
  Assign a variable \(b_{i} \in \mathbb{C} \setminus \set{0}\) (as in
  \cref{def:parameters}) to every segment of \(D\) (with the
  variables of the boundary segments matching as usual).
  Use these to assign each corner of the diagram a shape parameter
  \(z_{k}\) as in \cref{eq:shape parameters}; we assume that the
  \(b_{i}\) are chosen so that no \(z_{k}\) is \(1\).
  Labeling a crossing of sign \(\epsilon\) as in \eqref{eq:crossing
  labeled} the choice of \(m\) and \(b_{i}\) determine parameters
  \(a_{i}\) for each segment%
  \note{
    We denote these by \(a_{i}\) because they correspond to ratios of
    the region parameters \(a_{j}\).
  }
  by
  \begin{equation}
    \label{eq:segment a vars}
    \begin{aligned}
      a_{1}
      &=
      m^{\epsilon}
      \frac{
        1 - z_{\lW}
      }{
        1 - z_{\lN}
      }
      &
      a_{4}
      &=
      m^{-\epsilon}
      \frac{
        1 - z_{\lE}
      }{
        1 - z_{\lN}
      }
      \\
      a_{2}
      &=
      m^{-\epsilon}
      \frac{
        1 - z_{\lS}
      }{
        1 - z_{\lW}
      }
      &
      a_{3}
      &=
      m^{\epsilon}
      \frac{
        1 - z_{\lS}
      }{
        1 - z_{\lE}
      }
    \end{aligned}
  \end{equation}
  The construction above assigns every segment \emph{two} parameters
  \(a_{i}\) and \(a_{i}'\) from the crossings at both ends.
  The \defemph{segment equations} say that these parameters agree for
  each segment.
\end{definition}

The segment equations were first studied by
\textcite{arXiv:math/0009165} and the
associated octahedral decompositions by D.\ Thurston
\cite{ThurstonDNotes}. They are discussed in detail by
\textcite{Kim2018}. Our formulation follows the conventions of
\cite{McphailSnyder2022hyperbolicstructureslinkcomplements,McphailSnyder2024octahedralcoordinateswirtingerpresentation}.

\begin{proposition}
  \label{thm:segment solutions are reps}
  For an open knot diagram \(D\), consider the sets
  \begin{itemize}
    \item \(\mathsf{B}_{D}(m)\) of solutions of the segment equations of \(D\)
      with meridian eigenvalue \(m\)
    \item \(\mathsf{R}_{D}(m)\) of shadow-colored decorated representations
      \((\rho, u)\) with meridian eigenvalue \(m\) that are not \(D\)-pinched
  \end{itemize}
  Then the map \(\Phi : \mathsf{R}_{D}(m) \to \mathsf{B}_{D}(m)\)
  defined by \cref{eq:region parameters,eq:segment parameters} is surjective.
\end{proposition}

\begin{proof}
  \cite[Theorem
  2.11]{McphailSnyder2024octahedralcoordinateswirtingerpresentation}
  says that every point of \(\mathsf{R}_{D}(m)\) corresponds to an
  octahedral coloring, which
  gives a point of \(\mathsf{B}_{D}(m)\) after forgetting the region
  parameters \(a_{j}\).
  Conversely, if we have a point \(B \in \mathsf{B}_{D}(m)\), i.e.\ a
  solution of the segment equations with meridian eigenvalue \(m\),
  then by definition there is a well-defined assignment of parameters
  \(a_{i}\) to the segments of \(D\) that gives a non-pinched
  \defemph{octahedral coloring} of \(D\) in the sense of
  \cite{McphailSnyder2022hyperbolicstructureslinkcomplements}.
  (These \(a_{i}\) are simply the ratios of the adjacent region
  parameters \(a_{j}\).)
  By \cite[Theorem 2.11]{McphailSnyder2024octahedralcoordinateswirtingerpresentation}
  the holonomy of this octahedral coloring is \((\rho, u)\), which in particular means
  that \(\Phi(\rho, u) = B\).
\end{proof}

We note that this map is not a bijection; for example, a type (A) gauge
transformation by a unipotent matrix
\begin{equation*}
  \begin{bmatrix}
    1 & 0 \\
    \varphi & 1
  \end{bmatrix}
\end{equation*}
will preserve the parameters \(a_{j}, b_{i}, m\) but will in general change \((\rho, u)\).

\section{State integrals for the quantized Chern-Simons invariant}
\label{sec:State integrals big}

\subsection{Quantum dilogarithms}%
\label{sec:Quantum dilogarithms}

For a nonzero complex parameter \(\mathsf b\) \defemph{Faddeev's
noncompact quantum dilogarithm} \cite{Faddeev2001} is
\begin{equation}
  \Phi_{\mathsf b}(z) \defeq
  \exp
  \int_{\mathbb{R} + \ii \epsilon}
  \frac{
    \exp(-2 \ii zw)
  }{
    4 \sinh(w \mathsf b ) \sinh(w/\mathsf b)
  }
  \frac{dw}{w}
\end{equation}
for \(|\Im z| < |\Im c_{\mathsf b}|\), where
\(
  c_{\mathsf b} \defeq \frac{i}{2} \left( \mathsf b + \mathsf b^{-1}\right).
\)
\(\Phi_{\mathsf b}\) extends to a meromorphic function on
\(\mathbb{C}\) with an essential singularity at infinity.
We consider the case where \(\mathsf b = \sqrt \nr\) and the function
\(\pfl{}\) with
\begin{equation}
  \exp \pfl{t}
  =
  \Phi_{\sqrt \nr}\left(\ii \sqrt{\nr} t - c_{\sqrt \nr} +
  \frac{\ii}{\sqrt \nr} \right)
  .
\end{equation}
Strictly speaking \(\pfl{\zeta}\) depends on a choice of branch of
the logarithm but our formulas only depend on the well-defined
function \(e^{\pfl{\zeta}}\).
It is nonetheless convenient to use the logarithmic notation.
The recurrence relation
\begin{equation}
  \label{eq:pfl recurrence}
  \exp \pfl{t} =
  \frac{
    \exp \pfl{t-1/\nr}
  }{
    1 - e^{\tu t}
  }
\end{equation}
connects \(\pfl{}\) to the \(q\)-factorial and shows that the poles
and zeros of \(e^{\pfl{t}}\) all occur when \(t \in \frac{1}{\nr}
\mathbb{Z}\) and \(t \not \in (0,1)\).

\subsection{The action and state sums}
\label{sec:Definition of the invariant}

Following \cite[Section 6]{McPhailSnyderVolume} we define the action
functional of a diagram.
Consider the functions
\begin{equation}
  \label{eq:crossing qlog}
  \begin{aligned}
    \crossingfunction{+}{\mu}{t_{1}, t_{2}, t_{3}, t_{4}}
    &\defeq
    \pfl{t_{4} - t_{1}}
    +
    \pfl{t_{2} - t_{3}}
    \\
    &\phantom{\defeq}
    -
    \pfl{t_{2} - t_{1} - \mu/\nr + 1 - 1/\nr}
    -
    \pfl{t_{4} - t_{3} + \mu/\nr}
    \\
    &\phantom{\defeq}
    - \tu \mu(t_{2} + t_{3} - t_{1}- t_{4} - \mu/\nr)
    \\
    \crossingfunction{-}{\mu}{t_{1}, t_{2}, t_{3}, t_{4}}
    &\defeq
    -
    \pfl{t_{1} - t_{4} + 1 - 1/\nr}
    -
    \pfl{t_{3} - t_{2} + 1 - 1/\nr}
    \\
    &\phantom{\defeq}
    +
    \pfl{t_{1} - t_{2} + \mu/\nr}
    +
    \pfl{t_{3} - t_{4} - \mu/\nr + 1 - 1/\nr}
    \\
    &\phantom{\defeq}
    + \tu \mu(t_{2} + t_{3} - t_{1}- t_{4} - \mu/\nr)
  \end{aligned}
\end{equation}
These are logarithmic versions of the braiding kernels of
\cite[Section 5]{McPhailSnyderVolume}.
As discussed there \(t \mapsto \exp
\crossingfunction{\epsilon}{\mu}{t}\) is a quasi-periodic function of
each \(t_{i}\).
The periodicity constants are related to the segment equations of
\cref{sec:Segment equations}, i.e.\ to the gluing equations of the
octahedral decomposition.

Fix an oriented framed knot \(K\) and an open diagram \(D\) of \(K\).
Assign each internal segment a complex variable \(t_{i}\), assign the
boundary segments \(0\), and fix a complex number \(\mu\) (later, a
log-meridian of \(K\)).
Construct a function \(\actq{D,\mu}{t}\) of the variables \(t_{i}\)
by taking a sum over crossing functions
\begin{align}
  \label{eq:crossing positive}
  \begin{tikzpicture}[line width=1, baseline=10, scale=1]
    \draw[->] (0,0) node[left] {\(t_{2}\)} \br (1.5,1) node[right] {\(t_{4}\)};
    \draw[white, line width=10] (0,1) node[left] {} \br (1.5,0) node[right] {};
    \draw[->] (0,1) node[left] {\(t_{1}\)} \br (1.5,0) node[right] {\(t_{3}\)};
  \end{tikzpicture}
  \to
  \frac{1}{\nr}
  \crossingfunction{+}{\mu}{t_{1}, t_{2}, t_{3}, t_{4}}
  \\
  \label{eq:crossing negative}
  \begin{tikzpicture}[line width=1, baseline=10, scale=1]
    \draw[->] (0,1) node[left] {\(t_{1}\)} \br (1.5,0) node[right] {\(t_{3}\)};
    \draw[white, line width=10] (0,0) node[left] {} \br (1.5,1) node[right] {};
    \draw[->] (0,0) node[left] {\(t_{2}\)} \br (1.5,1) node[right] {\(t_{4}\)};
  \end{tikzpicture}
  \to
  \frac{1}{\nr}
  \crossingfunction{-}{\mu}{t_{1}, t_{2}, t_{3}, t_{4}}
\end{align}
(the factors of \(1/\nr\) are convenient when discussing asymptotics)
and assigning critical points shifts in the variables (which we could
also think of as \(\delta\) functions)
\begin{align}
  &
  \begin{tikzpicture}[line width=1, baseline=10, scale=1]
    \draw[->, looseness = 1.5] (0,0) node[right] {\(t\)} to [out
    =180, in=180] (0,1) node[right] {\(t\)};
  \end{tikzpicture}
  &
  &
  \begin{tikzpicture}[line width=1, baseline=10, scale=1]
    \draw[->, looseness = 1.5] (0,0) node[left] {\(t\)} to [out =0,
    in=0] (0,1) node[left] {\(t\)};
  \end{tikzpicture}
  &
  &
  \begin{tikzpicture}[line width=1, baseline=10, scale=1]
    \draw[<-, looseness = 1.5] (0,0) node[right] {\(t + 1/\nr - 1\)}
    to [out =180, in=180] (0,1) node[right] {\(t\)};
  \end{tikzpicture}
  &
  &
  \begin{tikzpicture}[line width=1, baseline=10, scale=1]
    \draw[<-, looseness = 1.5] (0,0) node[left] {\(t + 1 - 1/\nr\)}
    to [out =0, in=0] (0,1) node[left] {\(t\)};
  \end{tikzpicture}
\end{align}

\begin{definition}
  We call \(\actq{D, \mu}{}\) the \defemph{action} of the diagram.
  Note that it depends on the diagram \(D\) and on \(\mu\) but
  \emph{not} on a choice of representation.
\end{definition}

\begin{example}
  \label{ex:figure eight}
  Below is an open diagram \(D\) of the figure eight knot:
  
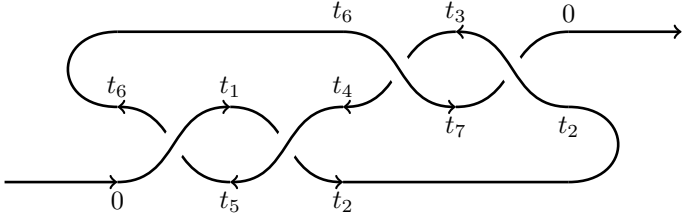
\begin{equation*}
    \begin{tikzpicture}[line width=1, scale=1, xscale=1.5]
      \draw[<-] (0,1) node[above] {\(t_{6}\)} \br (1,0) node[below] {\(t_{5}\)};
      \draw[white, line width=10] (0,0) node {} \br (1,1) node {};
      \draw[->] (0,0) node[below] {\(0\)} \br (1,1) node[above] {\(t_{1}\)};

      \draw[->] (1,1) node {} \br (2,0) node[below] {\(t_{2}\)};
      \draw[white, line width=10] (1,0) node {} \br (2,1) node {};
      \draw[<-] (1,0) node {} \br (2,1) node[above] {\(t_{4}\)};

      \draw[<-] (2,1) node {} \br (3,2) node {};
      \draw[white, line width=10] (2,2) node {} \br (3,1) node {};
      \draw[->] (2,2) node[above] {\(t_{6}\)} \br (3,1) node[below] {\(t_{7}\)};

      \draw[-] (3,1) node {} \br (4,2) node {};
      \draw[white, line width=10] (3,2) node {} \br (4,1) node {};
      \draw[<-] (3,2) node[above] {\(t_{3}\)} \br (4,1) node[below] {\(t_{2}\)};

      \draw (2,0) to (4,0) ;
      \draw[looseness = 1.5] (4,0) to [out =0, in=0] (4,1);

      \draw (0,2) to (2,2) ;
      \draw[looseness = 1.5] (0,1) to [out =180, in=180] (0,2);

      \draw[->] (-1,0) to (0,0);

      \draw[->] (4,2) node[above] {\(0\)} to (5,2);
    \end{tikzpicture}
  \end{equation*}
  Assigning variables to the segments as shown the action is
  \begin{align*}
    \actq{D, \mu}{t}
    &=
    \frac{1}{\nr}
    \crossingfunction{+}{\mu}{0, t_{5}, t_{1}, t_{6}}
    +
    \frac{1}{\nr}
    \crossingfunction{+}{\mu}{t_{4} + \tfrac{1}{\nr} - 1, t_{1}, t_{5}
    + \tfrac{1}{\nr} - 1, t_{2}}
    \\
    &\phantom{=}+
    \frac{1}{\nr}
    \crossingfunction{-}{\mu}{t_{3} + \tfrac{1}{\nr} - 1, t_{6}, t_{4}
    + \tfrac{1}{\nr} - 1, t_{7}}
    +
    \frac{1}{\nr}
    \crossingfunction{-}{\mu}{t_{7}, t_{2}, 0, t_{3}}
  \end{align*}
  The parameter shifts come from using critical points to rotate the
  crossings as in \cite[Figure 13]{McPhailSnyderVolume}.
\end{example}

Let \(\rho\) be a decorated representation of \(K\).
Choose an admissible, normalized shadow coloring \((\rho, u)\) of
\(D\) and a \defemph{log-meridian} \(\mu\), a logarithm of the
meridian eigenvalue \(m\) determined by the decoration.
(Recall that a normalized shadow coloring is one where the segment
parameter of the boundary segments is \(1\).)
Assume that the coloring is not pinched: if it is then \(\qinv{}\) is
still well-defined but we need to use a different expression for the
\(R\)-matrix coefficients.

Choose logarithms \(\beta_{i}\) of the segment parameters \(b_{i}\)
with \(e^{\tu \beta_{i}} = b_{i}\).
It is helpful to think of this choice as a vector \(\beta \in
\mathbb{C}^{E}\), where \(E\) is the set of internal segments.
We call \(\beta\) a \defemph{log-coloring}.
Since we are using a normalized coloring we can choose \(0\) as the
log-parameter of the boundary segments.
Re-writing \cite[Lemma 6.3]{McPhailSnyderVolume} in our normalizations gives

\begin{theorem}[\protect{\cite[Lemma 6.3]{McPhailSnyderVolume}}]
  \label{thm:state sum}
  The invariant \(\qinv{}\) defined in \cite{McPhailSnyderVolume} is
  given by the sum
  \begin{equation}
    \label{eq:state sum}
    \qinv{K, \rho, \mu}
    =
    \frac{1}{\nr^{|E|/2 + 1/2}}
    \sum_{n \in \nset{\nr}^{E}}
    e^{ \nr  \actq{D}{(\beta + n)/\nr} }
  \end{equation}
  where \(\nset{\nr} \defeq \set{0, 1, \dots, \nr -1}\).
  \( \qinv{K, \rho, \mu} \) is gauge-invariant, and in particular is
  independent of the choice of shadow coloring and log-coloring.
\end{theorem}

\begin{lemma}
  \label{thm:action lemmas}
  Let \(\beta\) be a log-coloring of a normalized, non-pinched shadow coloring.
  \begin{thmenum}
  \item
    \label{thm:action lemma:lattice periodic}
    \(t \mapsto \exp \nr \actq{D, \mu}{t}\) is a
    \(\mathbb{Z}^{E}\)-periodic function on the lattice
    \(
      \frac{\beta}{\nr} + \frac{1}{\nr}\mathbb{Z}^{E}
      .
    \)
  \item
    \label{thm:action lemma:no poles}
    If the shadow coloring avoids the unit circle, then none of the poles
    of the  function \(\mathbb{C}^{E} \to \mathbb{C}\) defined by
    \[
      t \mapsto \exp \nr \actq{D, \mu}{\beta/\nr + t}
    \]
    occur for \(t \in \mathbb{R}^{E}\).
    \qedhere
  \end{thmenum}

\end{lemma}

\begin{proof}
  \ref{thm:action lemma:lattice periodic}
  \Cref{eq:pfl recurrence} gives the quasi-periodicity relation
  \[
    e^{\pfl{\zeta/\nr + 1}}
    =
    \frac{
      e^{\pfl{\zeta/\nr}}
    }{
      (1 - e^{\tu (\zeta + 1)/\nr})
      \cdots
      (1 - e^{\tu (\zeta + \nr)/\nr})
    }
    =
    \frac{
      e^{\pfl{\zeta/\nr}}
    }{
      1 - e^{\tu \zeta}
    }
  \]
  so
  \begin{align*}
    \exp
    \crossingfunction{+}{\mu}{
      \frac{\beta_{1}}{\nr} + 1,
      \frac{\beta_{2}}{\nr},
      \frac{\beta_{1'}}{\nr},
      \frac{\beta_{2'}}{\nr}
    }
    &=
    \exp
    \crossingfunction{+}{\mu}{
      \frac{\beta_{1}}{\nr},
      \frac{\beta_{2}}{\nr},
      \frac{\beta_{1'}}{\nr},
      \frac{\beta_{2'}}{\nr}
    }
    m
    \frac{
      (1- b_{2'}/b_{1})
    }{
      (1- b_{2}/m b_{1})
    }
    \\
    &=
    a_{1}
    \exp
    \crossingfunction{+}{\mu}{
      \frac{\beta_{1}}{\nr},
      \frac{\beta_{2}}{\nr},
      \frac{\beta_{1'}}{\nr},
      \frac{\beta_{2'}}{\nr}
    }
  \end{align*}
  where \(a_{1}\) is the parameter assigned to the segment by
  \cref{eq:segment a vars}.
  Similarly
  \[
    \exp
    \crossingfunction{+}{\mu}{
      \frac{\beta_{1}}{\nr},
      \frac{\beta_{2}}{\nr},
      \frac{\beta_{1'}}{\nr} + 1,
      \frac{\beta_{2'}}{\nr}
    }
    =
    a_{1'}^{-1}
    \exp
    \crossingfunction{+}{\mu}{
      \frac{\beta_{1}}{\nr},
      \frac{\beta_{2}}{\nr},
      \frac{\beta_{1'}}{\nr},
      \frac{\beta_{2'}}{\nr}
    }
  \]
  and more generally the incoming segment variables have
  quasi-periodicity constant \(a_{i}\) and the outgoing have \(a_{i}^{-1}\).
  When the crossing functions \(\crossingfunction{\epsilon}{\mu}{}\)
  are assembled together into \(\actq{D,\mu}{}\) these factors
  cancel: the \(a_{i}\) match because the \(b_{i} = e^{\tu
  \beta_{i}}\) are a solution of the gluing equations.

  \ref{thm:action lemma:no poles}
  Recall that the zeros and poles of \(e^{\pfl{z}}\) all occur when
  \(z \in \frac{1}{\nr} \mathbb{Z}\).
  The arguments of the quantum dilogarithms \(\pfl{}\) in
  \cref{eq:crossing positive,eq:crossing negative}  are of the form
  \(\zeta/\nr + t_{i} -
  t_{j}\) with \(e^{\tu \zeta}\) the shape parameter of the relevant
  tetrahedron.
  Because \((\rho, u)\) avoids the unit circle \(\Im \zeta \ne 0\)
  for each tetrahedron so we do not encounter any singularities when
  \(t \in \mathbb{R}^{E}\).
\end{proof}

\subsection{State integrals}%
\label{sec:State integrals}

\begin{definition}
  \label{def:state integral}
  The \defemph{state integrals} associated to a log-coloring \(\beta\) are
  \begin{equation}
    \label{eq:state integral again}
    \stint{D, \mu, \beta}{k} \defeq
    \int_{[0,1]^{E}}
    \exp \nr \left( \actq{D, \mu}{t + \beta/\nr} - \tu k \cdot t \right)
    \dif t
  \end{equation}
  for each integer vector \(k \in \mathbb{Z}^{E}\).
\end{definition}

\begin{bigtheorem}
  \label{main result}
  Let \(D\) be an open diagram of the knot \(K\).
  Let \((\rho, u)\) be an admissible, normalized, non-pinched shadow coloring that
  avoids the unit circle.
  Choose a log-coloring \(\beta \in \mathbb{C}^{E}\) corresponding to \((\rho, u)\), i.e.\ choose logarithms \(\beta_{i}\) of the segment parameters \(b_{i}\) of \((\rho, u)\).
  Then \(\qinv{}\) is the symmetric
  rectangular sum of the state integrals
  \begin{equation}
    \label{eq:state integral sum}
    \qinv{K, \rho, \mu}
    =
    \nr^{|E|/2 - 1/2}
    \lim_{\min_{i\in E}M_i\to\infty}
    \sum_{\substack{k\in\mathbb Z^E\\ |k_i|\leq M_i\ (i\in E)}}
    \stint{D, \mu, \beta}{k},
  \end{equation}
  with convergence in the Pringsheim sense.
\end{bigtheorem}

\begin{proof}
  By \cref{thm:action lemmas}\ref{thm:action lemma:no poles} the function
  \[
    F(t) \defeq \exp\left(\nr\actq{D,\mu}{t+\beta/\nr}\right).
  \]
  is smooth on \(\mathbb{R}^E\).  Let \(f\) be the periodic
  function obtained by restricting \(F\) to the
  half-open cube \([0,1)^E\), and write
  \[
    a_k
    \defeq
    \int_{[0,1]^E} f(t) e^{-\tu k \cdot t} \dif t
  \]
  for its Fourier coefficients.  Thus
  \(a_{\nr k}=\stint{D, \mu, \beta}{k}\).
  Consider the selected-subseries function
  \begin{equation}
    \label{eq:selected subseries function}
    (P_{\nr}f)(t)
    \defeq
    \frac{1}{\nr^{|E|}}
    \sum_{n\in[\nr]^E}
    f\left(t + \frac{n}{\nr}\right).
  \end{equation}
  As we show below it is obtained from \(f\) by dropping all the Fourier
  coefficients whose frequencies are not multiples of \(\nr\).
  We can interpret \cref{thm:state sum} as saying that
  \begin{equation*}
    \qinv{K, \rho, \mu}
    =
    \nr^{(|E| - 1)/2}
    (P_{\nr}f)(0)
  \end{equation*}
  so our claim reduces to showing that
  \begin{equation}
    \label{eq:selected subseries function fourier}
    (P_{\nr}f)(0)
    =
    \sum_{k \in \mathbb{Z}^{E}} a_{\nr k}
  \end{equation}
  in the sense of Pringsheim convergence (as in the theorem statement).
  The sum in \eqref{eq:selected subseries function fourier} is the
  Fourier series of \(P_{\nr}f\) at \(0\).
  To see this, compute the \(\ell\)th Fourier coefficient by changing
  variables to \(y = t + n/\nr\) and using the periodicity of \(f\):
  \begin{align*}
    \widehat{P_{\nr}f}(\ell)
    &=
    \frac{1}{\nr^{|E|}}
    \sum_{n\in[\nr]^E}
    \int_{[0,1]^E} f\leftfun(t + \frac{n}{\nr}\rightfun)
    e^{-\tu \ell\cdot t}\dif t
    \\
    &=
    \frac{1}{\nr^{|E|}}
    \sum_{n\in[\nr]^E}
    \int_{n/ \nr +[0,1]^E} f(y)e^{-\tu \ell \cdot ( y - n/\nr) }\dif y
    \\
    &=
    \frac{1}{\nr^{|E|}}
    \sum_{n\in[\nr]^E}
    \int_{[0,1]^E} f(y)e^{-\tu \ell \cdot ( y - n/\nr) }\dif y
    \\
    &=
    \int_{[0,1]^E} f(y)e^{-\tu \ell \cdot  y }\dif y
    \times
    \frac{1}{\nr^{|E|}}
    \sum_{n\in[\nr]^E}
    e^{\tu \ell \cdot n/\nr}
    \\
    &=
    \begin{cases}
      a_{\ell} & \ell \in (\nr \mathbb{Z})^{E}
      \\
      0 & \text{otherwise.}
    \end{cases}
  \end{align*}

  Thus, our goal is to show that the Fourier series
  \eqref{eq:selected subseries function fourier} of \(P_{\nr}f\)
  converges to \((P_{\nr}f)(0)\) at \(0\).
  First we show that \(P_{\nr} f\) is continuous at \(0\).
  \Cref{thm:action lemmas}\ref{thm:action lemma:lattice periodic}
  says that \(F\) is periodic on the lattice points \(n/\nr, n \in
  \mathbb{Z}^{E}\),
  which means that \(f\) is continuous at these points.
  In particular each summand defining \(P_{\nr}f\) is continuous at \(0\),
  so it is as well.

  To conclude the proof we need to show that \eqref{eq:selected
  subseries function fourier} converges in the Pringsheim sense.
  We use the following result on convergence of multiple Fourier series, which
  \textcite{TelyakovskiiTemlyakov1997} attribute to
  \textcite{zbMATH02650645,zbMATH03060813}.
  \begin{lemma}
    \label{thm:convergence lemma}
    Let \(g : (\mathbb{R}/\mathbb{Z})^{E} \to \mathbb{C}\) be a
    periodic function with Fourier coefficients \(c_{k}\).
    If \(g\) has bounded Hardy-Krause variation and \(g\) is continuous%
    \note{
      More generally the series converges to the average of the
      orthant limits of \(g\) at \(x\).
    }
    at \(x\) then
    \begin{equation*}
      g(x) =
      \lim_{\min_{i\in E}M_i\to\infty}
      \sum_{\substack{k\in\mathbb Z^E\\ |k_i|\leq M_i\ (i\in E)}}
      c_{k} e^{\tu k \cdot x}
      .
      \qedhere
    \end{equation*}
  \end{lemma}

  We saw above that \(P_{\nr} f\) is continuous at \(0\) so it
  remains only to show that \(P_{\nr}f\) has bounded Hardy variation.
  To define this condition we first consider the Vitali variation.
  Suppose \(h\) is a function on \([0,1]^{d}\).
  Choose partitions \(R = R_{1}, \dots, R_{d} \) of each
  coordinate interval with points
  \(0 = a_{0}^{j} < a_{1}^{j} \cdots < a_{P_{j}}^{j} = 1\).
  The difference over the \(j\)th coordinate is
  \begin{equation*}
    \Delta_{n}(j, R) h \defeq
    h(x_{1}, \dots, x_{j-1}, a_{n}^{j}, x_{j+1}, \dots)
    -
    h(x_{1}, \dots, x_{j-1}, a_{n-1}^{j}, x_{j+1}, \dots)
  \end{equation*}
  and the total difference with respect to \(R\) is the sum over the composition
  \begin{equation*}
    \operatorname{VarV}(h, R)
    =
    \sum_{1 \le n_{j} \le M_{j}}
    | \Delta_{n_{1}}(1, R) \cdots \Delta_{n_{d}}(d, R) h |
  \end{equation*}
  of the differences. The \defemph{Vitali variation} of \(h\) is the
  supremum over all partitions:
  \begin{equation*}
    \operatorname{VarV}(h)
    \defeq
    \sup_{R} \operatorname{VarV}(h, R)
  \end{equation*}

  Now consider a function \(g : [0,1]^{E} \to \mathbb{C}\). For any
  subset \(B \subset E\) let \(g^{B}\) be the restriction to the
  variables indexed by \(B\).
  Roughly speaking \(g\) has bounded Hardy-Krause (HK) variation if
  the Vitali variations of all the restrictions \(g^{B}\) are finite.
  In particular this accounts for jumps of \(g\) across faces of
  every codimension.
  The original definition of HK variation (as given in
  \cite{TelyakovskiiTemlyakov1997}) requires the restricted Vitali
  variations to be bounded for \emph{every} value of the remaining
  non-fixed variables (indexed by \(E \setminus B\)).
  One can instead consider the anchored variation where the values of
  the \(E \setminus B\) variables are fixed, say to \(1\).
  \textcite{Owen2005} notes that to show the HK variation is bounded
  in this sense it suffices to consider bounds on the anchored HK
  variation and attributes this equivalence to \textcite{Young1913}.

  We now give the relevant definition.
  For each subset \(B \subset E\), let
  \[
    g^B(x)=g(x_B,1_{E \setminus B}),
    \qquad x\in[0,1]^B,
  \]
  be the restriction of \(g\) to the upper coordinate face anchored
  at \(1\).  The \defemph{Hardy-Krause variation anchored at \(1\)} is
  \[
    \operatorname{VarHK}(g)
    =
    \sum_{\emptyset \ne B\subset E}\operatorname{VarV}(g^B).
  \]
  By the discussion above the theorem follows from showing that
  \(
    \operatorname{VarHK}(P_{\nr} f) < \infty.
  \)

  Writing \(T_{s}f(t) = f(t + s)\) for translates it is clear that
  \begin{equation*}
    \operatorname{VarHK}(P_{\nr}f)
    \le
    \frac{1}{\nr^{|E|}}
    \sum_{k \in [\nr]^{E}}
    \operatorname{VarHK}(T_{n/\nr}f)
  \end{equation*}
  so it suffices to find a uniform bound on the variations of the
  translates of \(f\).
  For \(B\subseteq E\), write
  \(\partial_B F=\prod_{i\in B}\partial_{t_i} F\)
  and set
  \[
    C_B=\sup_{[0,1]^E}|\partial_B F|<\infty.
  \]
  We show that
  \begin{equation}
    \label{eq:HV sufficient condition}
    \operatorname{VarV}(T_{s}f^{B}, R)
    \leq \sum_{J\subseteq B}2^{|J|}C_{B\setminus J}
  \end{equation}
  for any translate \(T_{s} f\) and nonempty \(B \subseteq E\).
  We think of this sum as consisting of a ``smooth part'' \(J = \emptyset\)
  and additional terms that account for the jumps caused by
  discontinuities of \(f\) at the faces.
  When summed over all \(B\) the smooth parts give the standard
  estimate of Vitali variation by the \(L^1\)-norm of the full mixed
  derivative; see
  \cite[Proposition 13]{Owen2005} and
  \cite[equations (2--3)]{BasuOwen2016}.

  To prove \eqref{eq:HV sufficient condition} consider first the case
  where \(B = \set{1}\) is the
  variation in a single coordinate \(t_{1}\).
  The fundamental theorem of calculus expresses each signed increment
  as the integral of the corresponding derivative over the smooth
  pieces, together with a jump whenever the interval crosses the periodic seam.
  At that seam the jump is the difference between the values of
  \(F\) at \(t_{1} = 0\) and \(t_{1} = 1\).
  We conclude
  \[
    \operatorname{VarV}(T_s f^{\{1\}},R)
    \leq \int_0^1 |\partial_1F(x,y_s)|\dif x
    + |F(0,y_s)-F(1,y_s)|
    \leq C_{\{1\}}+2C_{\emptyset},
  \]
  where \(y_s\in[0,1)^{E\setminus\{1\}}\) consists of the anchored
  complementary coordinates after translation by \(s\) and reduction
  modulo one.
  Here \(C_{\set{1}}\) is the smooth part and \(C_{\emptyset}\) is
  the jump contribution.
  In this case the jump contribution vanishes because we are evaluating at
  lattice points (we only need translations by \(\tfrac{1}{\nr}\mathbb{Z}\)).

  For the general case apply the increment formula in every coordinate of \(B\).
  Expanding the product gives one term for each subset
  \(J\subseteq B\): take a jump in the coordinates of \(J\)
  and integrate a derivative in the remaining coordinates.
  Since we are integrating over non-lattice points the jump contributions do not vanish.
  The iterated jump is an alternating sum of \(2^{|J|}\)  integrals
  of \(\partial_{B\setminus J}F\) so its absolute value
  is at most \(2^{|J|}C_{B\setminus J}\).
  Summing over the partition rectangles gives \eqref{eq:HV sufficient
  condition}.
\end{proof}

\begin{remark}
  \label{rem:absolute convergence}
  The construction in the proof shows that we should not expect the
  series in \cref{main result} to converge absolutely. If it did, then
  the  Fourier series of \(P_{\nr}f\) would converge absolutely, so
  \(P_{\nr} f\) would be  continuous. However, there is no reason to expect this
  because \(P_{\nr} f\) ought to be discontinuous across the boundary
  of the fundamental domain for the same reasons that \(f\) is.
\end{remark}

\section{Formal asymptotics}%
\label{sec:Formal asymptotics}

We conclude by formally (i.e.\ non-rigorously) studying the large
\(\nr\) asymptotics of the
state integrals and discussing barriers to rigorously establishing them.
Our discussion suggests that new ideas are needed to understand the
cancellation of boundary terms in the state integrals.

\subsection{Classical dilogarithms}%
\label{sec:Classical dilogarithms}

The computation of Chern-Simons invariants is closely related to a
special function called the \defemph{dilogarithm} defined by
\[
  \operatorname{Li}_{2}(z)
  \defeq
  - \int_{0}^{z} \frac{\log(1 - t)}{t} \, d t
\]
We refer to \textcite{Zagier2007} for more information on the dilogarithm.
It has a branch point at \(1\) and crossing the branch cut picks up a
\(\log(z)\) term, which itself has a branch point at \(0\).
Analytic continuation of \(e^{\operatorname{Li}_{2}(z)}\) leads to
the flattenings of \textcite{Neumann2004}.
To avoid dealing with these in full generality, write \(X = (-\infty,
0] \cup [1, \infty)\) and consider \cite[Chapter 4]{Murakami2018} the
holomorphic function \(\dll{} \colon \mathbb{C} \setminus X \to
\mathbb{C}\) defined by
\begin{equation}
  \dll{\zeta}
  \defeq
  \begin{cases}
    \displaystyle
    \frac{
      \operatorname{Li}_{2}(e^{\tu \zeta})
    }{
      \tu
    }
    &
    \Im \zeta > 0
    \\
    \displaystyle
    -
    \frac{
      \operatorname{Li}_{2}(e^{-\tu \zeta})
    }{
      \tu
    }
    - \pi \ii \zeta(\zeta -1)
    - \frac{\pi \ii}{6}
    &
    \Im \zeta < 0
  \end{cases}
\end{equation}

\begin{lemma}
  \label{thm:dilog lemma}
  For fixed \(\zeta \in \mathbb{C} \setminus X \),
  \[
    \exp \frac{ \mathrm{d} \dll{\zeta}}{\mathrm{d} \zeta}
    =
    \frac{
      1
    }{
      1 - e^{\tu \zeta}
    }
  \]
  and
  \[
    e^{\pfl{\zeta}}
    =
    \exp\leftfun[
      \nr
      \dll{\zeta + 1/2\nr}
      +
      \bigoh(\nr^{-1})
    \rightfun]
    \qedhere
  \]
\end{lemma}

\begin{proof}
  The first claim is an elementary computation.
  In the notation of \textcite{Murakami2018} we have \(
  e^{\pfl{\zeta}} = \psi_{\nr}(\zeta + 1/2\nr) \) so the second claim
  follows from \cite[eq.\ (4.3)]{Murakami2018}.
\end{proof}

\subsection{The classical action}%
\label{sec:The classical action}

Here we define an action \(\act{D}{}\) corresponding to the \(\nr \to
\infty\) limit of \(\actq{D}{}\).
Set
\begin{equation}
  \label{eq:crossing log}
  \begin{aligned}
    \crossingfunctionclass{+}{t_{1}, t_{2}, t_{3}, t_{4}}
    &\defeq
    \dll{t_{4} - t_{1}}
    +
    \dll{t_{2} - t_{3}}
    \\
    &\phantom{\defeq}
    -
    \dll{t_{2} - t_{1}  + 1 }
    -
    \dll{t_{4} - t_{3} }
    \\
    \crossingfunctionclass{-}{t_{1}, t_{2}, t_{3}, t_{4}}
    &\defeq
    -
    \dll{t_{1} - t_{4} + 1}
    -
    \dll{t_{3} - t_{2} + 1}
    \\
    &\phantom{\defeq}
    +
    \dll{t_{1} - t_{2}}
    +
    \dll{t_{3} - t_{4} + 1}
  \end{aligned}
\end{equation}
Assign complex variables \(t_{i}\) to the segments as before and
construct a function \(\act{D}{t}\) by taking a sum over crossing functions
\begin{align}
  \label{eq:crossing positive class}
  \begin{tikzpicture}[line width=1, baseline=10, scale=1]
    \draw[->] (0,0) node[left] {\(t_{2}\)} \br (1.5,1) node[right] {\(t_{4}\)};
    \draw[white, line width=10] (0,1) node[left] {} \br (1.5,0) node[right] {};
    \draw[->] (0,1) node[left] {\(t_{1}\)} \br (1.5,0) node[right] {\(t_{3}\)};
  \end{tikzpicture}
  \to
  \crossingfunctionclass{+}{t_{1}, t_{2}, t_{3}, t_{4}}
  \\
  \label{eq:crossing negative class}
  \begin{tikzpicture}[line width=1, baseline=10, scale=1]
    \draw[->] (0,1) node[left] {\(t_{1}\)} \br (1.5,0) node[right] {\(t_{3}\)};
    \draw[white, line width=10] (0,0) node[left] {} \br (1.5,1) node[right] {};
    \draw[->] (0,0) node[left] {\(t_{2}\)} \br (1.5,1) node[right] {\(t_{4}\)};
  \end{tikzpicture}
  \to
  \crossingfunctionclass{-}{t_{1}, t_{2}, t_{3}, t_{4}}
\end{align}
and assigning critical points shifts in the variables (which we could
also think of as \(\delta\) functions)
\begin{align}
  &
  \begin{tikzpicture}[line width=1, baseline=10, scale=1]
    \draw[->, looseness = 1.5] (0,0) node[right] {\(t\)} to [out
    =180, in=180] (0,1) node[right] {\(t\)};
  \end{tikzpicture}
  &
  &
  \begin{tikzpicture}[line width=1, baseline=10, scale=1]
    \draw[->, looseness = 1.5] (0,0) node[left] {\(t\)} to [out =0,
    in=0] (0,1) node[left] {\(t\)};
  \end{tikzpicture}
  &
  &
  \begin{tikzpicture}[line width=1, baseline=10, scale=1]
    \draw[<-, looseness = 1.5] (0,0) node[right] {\(t - 1\)} to [out
    =180, in=180] (0,1) node[right] {\(t\)};
  \end{tikzpicture}
  &
  &
  \begin{tikzpicture}[line width=1, baseline=10, scale=1]
    \draw[<-, looseness = 1.5] (0,0) node[left] {\(t + 1\)} to [out
    =0, in=0] (0,1) node[left] {\(t\)};
  \end{tikzpicture}
\end{align}

\begin{definition}
  The construction above defines the \defemph{classical action} as a
  holomorphic function
  \[
    \act{D}{} \colon P_{D} \to \mathbb{C}
  \]
  where \(P_{D}\) is the set of points \(t \in \mathbb{C}^{E}\)
  where the arguments of the dilogarithms \eqref{eq:crossing log}
  (adjusted by critical point shifts)
  avoid the set \(X\).
\end{definition}

The classical action is a variant of Yokota's \defemph{potential
function} introduced in \cite{arXiv:math/0009165} and used throughout
the Volume Conjecture literature.
We prefer to call it an \emph{action} to match the language used in
physics, since it is essentially the Chern-Simons action if we view
\(\Omega_{D}\) as parametrizing flat \(\mathfrak{sl}_{2}\)
connections on the knot complement.
Our action includes parameter shifts that appear to be related to the
normalization terms \(\Omega_{1}, \Omega_{2}\) studied by
\textcite{zbMATH06433029}.

\begin{remark}
  \Cref{thm:dilog lemma} shows that \(\act{D}{}\) is the
  pointwise limit of \(\actq{D}{}\).
  Justifying a saddle-point approximation requires upgrading this to
  uniform convergence over the integration contour.
  If the contour stays away from the \(\nr\)-dependent singular
  region this is not difficult to show because our integrands are meromorphic
  and converging pointwise to a holomorphic function.
  Unfortunately the contours appearing in our state integrals do
  \emph{not} avoid the singularities as \(\nr \to \infty\).
  We discuss this further in \cref{sec:Barriers to Step 2}.
\end{remark}

\begin{definition}
  A \defemph{generalized critical point} of the classical action is a point
  \(\gamma \in \mathbb{C}^{E}\) together with a \defemph{flattening
  vector} \(k \in \mathbb{Z}^{E}\) so that
  \[
    \left.
    \frac{
      \partial \act{D}{}
    }{
      \partial t_{i}
    }
    \right|_{\gamma}
    =
    \tu
    k_{i}
    \text{ for all } i \in E.
  \]
  Such a point has
  \[
    \exp
    \left.
    \frac{
      \partial \act{D}{}
    }{
      \partial t_{i}
    }
    \right|_{\gamma}
    =
    1
    \text{ for all } i \in E.
  \]
\end{definition}

\begin{proposition}
  \label{thm:crit points are reps}
  Let \(D\) be a diagram of a knot \(K\).
  \begin{thmenum}
  \item \label{thm:crit points are reps:correspondence}
    Using the parameters \(\gamma\) to determine a log-coloring of
    \(D\) gives a normalized,
    non-pinched shadow coloring \((\rho, u)\) of \(D\) with meridian
    eigenvalue \(1\).
  \item \label{thm:crit points are reps:values}
    If \((\gamma, k)\) is a generalized critical point corresponding
    to \((\rho, u)\), then its critical value is related to the \(\slg\)
    Chern-Simons invariant of \(\rho\) by
    \begin{equation}
      \label{eq:critical value shift}
      \mathcal{I}^{\psi}(K, \rho)
      =
      \exp \leftfun( \act{D}{\gamma}  - \tu k \cdot \gamma \rightfun)
      .
    \end{equation}
  \end{thmenum}
  Here \(\mathcal{I}^{\psi}(K, \rho)\) is the \(\slg\)
  Chern-Simons invariant in the normalization of \cite{McPhailSnyderVolume},
  which is independent of \(u\).
\end{proposition}

\begin{proof}
  \ref{thm:crit points are reps:correspondence}
  \Cref{thm:dilog lemma} shows that generalized critical points are
  equivalent to the segment equations for \(m = 1\), so we can apply
  \cref{thm:segment solutions are reps}.

  \ref{thm:crit points are reps:values}
  The second claim was worked out in detail by
  \textcite{zbMATH06370282}.
  We explain the idea in our notation.
  In \cite{McPhailSnyderVolume} the Chern-Simons invariant
  \(\log \mathcal{I}^{\psi}\) is defined by taking a sum of
  contributions over crossings.
  However, instead of using the crossing functions \eqref{eq:crossing log}
  defined in terms of \(\dll{}\) one must use lifted dilogarithms
  \begin{equation}
    \label{eq:lifted dilog}
    \lambda(\zeta)
    =
    \dll{\zeta} + \zeta \log(1 - e^{\tu \zeta})
    =
    \dll{\zeta} - \zeta \frac{\partial \dll{\zeta}}{\partial \zeta}
    .
  \end{equation}
  One explanation is that we need to extend \(\dll{}\) to a
  suitable covering space to handle the branch cuts.
  In the notation of \cite{zbMATH06370282} this corresponds to using
  \(V_{\mathbf{0}}\) and not \(V\).

  As a consequence the Chern-Simons invariant is given not by evaluating
  \(\act{D}{}\) at its critical points but by evaluating a related function
  \begin{equation*}
    \actadj{D}{t}
    =
    \act{D}{t}
    -
    \sum_{i \in E}
    t_{i}
    \sigma_{i}(t)
  \end{equation*}
  where
  \begin{equation*}
    \sigma_{i}(t)
    =
    \left[
      \pm
      \log(1 - e^{\pm \tu(t_{i} - t_{j_{k}})} )
    \right]
  \end{equation*}
  is a sum of terms corresponding to the four segments \(j_{k}\)
  adjacent to \(i\).
  The signs depend on the combinatorics of the diagram: whether the
  segment is incoming or outgoing, the crossings signs, and so on.
  If the \(t_{i}\) correspond to a log-coloring (i.e.\ if the \(e^{\tu t_{i}}\) 
  are a solution of the segment equations) then the \(\sigma_{i}(t)\)
  are logarithms of the ratios of region parameters occurring in  \eqref{eq:segment a vars}.

  In general the \(\sigma_{i}\) are complex-valued.
  However, if \(t = \gamma\) is a generalized critical point then
  \begin{equation*}
    \sigma_{i}(\gamma) = \tu k_{i}
  \end{equation*}
  is the flattening vector of \(\gamma\) by \eqref{eq:lifted dilog}.
  Thus the Chern-Simons invariant is
  \begin{equation*}
    \actadj{D}{\gamma}
    =
    \act{D}{\gamma}
    -
    \sum_{i \in E}
    \gamma_{i}
    \sigma_{i}(\gamma)
    =
    \act{D}{\gamma} - \tu k \cdot \gamma.\qedhere
  \end{equation*}
\end{proof}

\begin{remark}
  \label{rem:flattening vector}
  \textcite{Neumann2004} gave an algebraic method to compute the
  \(\slg\) Chern-Simons invariant of a triangulated manifold.
  The hyperbolic structure on \(M\) is encoded using a solution of
  Thurston's gluing equations as usual, but his method requires an
  additional choice of combinatorial data called a \defemph{flattening}.
  This is a coherent choice of logarithms of the shape parameters
  satisfying a logarithmic form of the gluing equations.
  For an octahedral decomposition the only nontrivial flattening
  equations come from the segment equations and are of the form
  \[
    \pm \mu + \log(1 - z_{j}) -  \log(1 - z_{j'})
    =
    \pm \mu + \log(1 - z_{l}) -  \log(1 - z_{l'})
    + k_{i}
  \]
  for an integer \(k_{i}\) at each segment.
  These are precisely the integers appearing in the flattening vector above.
\end{remark}

By work of \textcite{Francaviglia2004}
there is a well-defined notion of hyperbolic volume \(\vol(\rho)\)
for any representation \(\rho\) of a knot group into \(\slg\).
For boundary-parabolic%
\note{
  For general \(\rho\) there is an elementary correction term.
  We can think of this as being related to Dehn filling.
}
representations \(\vol(\rho)\) is equal to the real part of the
\(\slg\) Chern-Simons invariant (times \(1/2\pi\) in our normalization).
Together with volume rigidity this gives:

\begin{corollary}
  \label{thm:hyp is max critical value}
  Let \(K\) be a hyperbolic knot and let \(\rho_{\hyp}\) be a
  representation corresponding to its canonical hyperbolic structure,
  lifted to \(\slg\) with meridian eigenvalue \(1\).
  Suppose \(D\) is a diagram of \(K\) and that \(\rho_{\hyp}\) is not
  \(D\)-pinched.
  \begin{thmenum}
  \item
    \label{thm:hyp is max critical value:exists}
    There is a generalized critical point \(\gamma\)
    corresponding to \((\rho, u)\) for some admissible shadow coloring \(u\).
  \item
    \label{thm:hyp is max critical value:real part}
    Let \(\varsigma_{K}\) be the Chern-Simons
    invariant of \(K\) (i.e.\ of its complete hyperbolic structure)
    in the normalization of \cite{McPhailSnyderVolume}.
    Then
    \(
      \Re
      \varsigma_{K}
      =
      \vol(\rho_{\hyp})
      /
      2 \pi
      =
      \vol(K)
      /
      2 \pi
      .
    \)
  \item
    \label{thm:hyp is max critical value:maximal}
    Any generalized critical point \((\gamma', k')\) has
    \begin{equation*}
      \Re \left[ \act{D}{\gamma'} - \tu \gamma' \cdot k'\right]
      \le
      \Re \left[ \act{D}{\gamma} - \tu \gamma \cdot k\right]
    \end{equation*}
    with equality if and only if the underlying representation of
    \(\gamma'\) is \(\rho_{\hyp}\).\qedhere
  \end{thmenum}
\end{corollary}

\begin{proof}
  \ref{thm:hyp is max critical value:exists} is a special case of
  \cref{thm:crit points are reps}.
  \ref{thm:hyp is max critical value:real part} is the correct
  normalization (see \cite[Example 4.6]{McPhailSnyderVolume}) of the
  well-known relation between the hyperbolic volume and the
  Chern-Simons invariant, which was first proved by \textcite{Yoshida1985}.
  \ref{thm:hyp is max critical value:maximal} is volume rigidity for
  representations \cite[Theorem 7.2]{Francaviglia2004} re-stated in
  terms of critical points.
\end{proof}

\subsection{Asymptotics}%
\label{sec:Asymptotics}

It is immediate from \cref{thm:dilog lemma} that \(\actq{D}{}\)
converges (pointwise) on \(\Omega_{D}\) to \(\act{D}{}\) as \(\nr \to
\infty\), which suggests we can approximate the state integrals
\(\stint{D, \mu, \beta}{k}\) of \eqref{eq:state integral again} by
using a saddle-point approximation.
\citeauthor{zbMATH06684922} gives a detailed explanation of the
saddle point approximation in the context of the Volume Conjecture in
\cite[Section 3]{zbMATH06684922}, which we summarize here.

Let \(f\) be holomorphic on some domain in  \(\mathbb{C}^{\ell}\)
containing an isolated critical point \(\gamma\).
A contour \(\Gamma \colon [-1,1]^{\ell} \to \mathbb{C}^{\ell}\) with
\(\Gamma(0) = \gamma\) passes through \(\gamma\) in \defemph{saddle
point position} if \(\Re f(\Gamma(t)) < \Re f(\gamma) \) except for
when \(t = 0\) and it satisfies an additional orientation condition.%
\note{
  For example, when \(\ell = 1\) the orientation condition says that
  \(\Gamma\) should go from one side of the mountain pass to the
  other, not double back.
}
The saddle point approximation says that
\[
  \int_{\Gamma} e^{\nr f(t)} \dif t
  =
  \nr^{-\ell/2}
  e^{\nr f(\gamma)}
  [ \tau + \bigoh(\nr^{-1}) ]
\]
for a constant \(\tau\) related to the determinant of the Hessian of
\(f\) at \(\gamma\), which we require to be nondegenerate.
It can be refined to include higher-order terms in the exponent whose
coefficients are related to derivatives of \(f\).
The idea is that near the critical point we can approximate \(e^{\nr
f}\) by a Gaussian.

We can apply this approximation to study the asymptotics of
\(\qinv{K, \rho, u}\) for \(K\) a hyperbolic knot.
Suppose there is some diagram \(D\) of \(K\) where \(\rho\) is not pinched;
as discussed in \cref{sec:Gauge transformations and admissibility} we
can think of this as a nondegeneracy condition on \(\rho\).
By \cref{main result} the invariant \(\qinv{K, \rho, u}\) is equal to
a sum of contour integrals.
We hope that it is valid in the large \(\nr\) limit to replace them by
modified integrals
\[
  \stintadj{D}{k} \defeq
  \int_{\Gamma}
  \exp \nr \left( \act{D}{t} - \tu k \cdot t \right)
  \dif t
\]
for some other contour \(\Gamma\).
The dependence on the initial representation \(\rho\) drops out
because \(\beta/\nr \to 0\); this is one explanation for why the
growth rate in \eqref{eq:asymptotics sign positive} is independent of \(\rho\).
Changing the contour \(\Gamma\) is plausible because \(\act{D}{}\) is
a holomorphic function and we can use Stokes' Theorem (i.e.\ a
higher-dimensional Cauchy theorem) but there are other issues
discussed in \cref{sec:Barriers to Step 2}.

Let us make this more precise. Since we assumed \(\rho\) was not pinched
(for the diagram \(D\)) the complete hyperbolic
structure \(\rho_{\hyp}\) is also not pinched.
Write \(\varsigma_{K}\) for the complex Chern-Simons
invariant of \(K\), which has \(2\pi \Re \varsigma_{K} = \vol(K)\).
By \cref{thm:hyp is max critical value} there is a generalized critical
point \(\gamma\) of \(\act{D}{}\) with flattening vector \(k_{0}\)
corresponding to some normalized shadow coloring of \(\rho_{\hyp}\)
avoiding the unit circle.
Moving the contour (\cref{step 2}) means showing it is valid to write
\begin{equation}
  \qinv{K, \rho, \mu}
  =
  \nr^{|E|/2 - 1/2}
  \sum_{k \in \mathbb{Z}^{E}}
  \stintadj{D}{k}
  +
  \bigoh(e^{\nr(\Re \varsigma_{K} - \delta)})
\end{equation}
for some \(\delta > 0\).
(We are working formally, but it would be enough for this sum to
converge as in \cref{main result}, or even in the Ces\`{a}ro sense.)

We now explain why this would lead to the expected asymptotics.
Suppose that \(\Gamma\) passes through \(\gamma\) in saddle point
position (which is not quite right: see below).
Then the saddle point approximation gives
\[
  \nr^{|E|/2 - 1/2}
  \stintadj{D}{k_{0}}
  =
  \nr^{-1/2}
  e^{\nr \varsigma_{K}}
  \left[
    \tau
    +
    \bigoh\left(\nr^{-1}\right)
  \right]
\]
because \(\Gamma\) is an \(|E|\)-dimensional contour.
We do not expect the terms of the sum with \(k \ne k_{0}\) to contribute at
highest order to the asymptotics because the saddle point
approximation also requires the imaginary part of the derivative of
\(\act{D}{}\) to vanish. If it does not, the growth is suppressed by
rapid oscillations in the phase of the integrand.
The sum over \(k\) in \eqref{eq:state integral sum} naturally captures
the extra choice of flattening used in the computation of \(\varsigma_{K}\).
We do not expect any contributions from other critical points of \(\act{D}{}\)
because of \cref{thm:hyp is max critical value}\ref{thm:hyp is max
critical value:maximal}.
If \((\gamma', k_{0}')\) is another, distinct critical point its saddle-point
contribution would be of order \(\exp \nr(\act{D}{\gamma'} - \tu
\gamma' \cdot k' )\) or smaller and
\begin{equation*}
  \Re \leftfun[ \act{D}{\gamma'} - \tu \gamma' \cdot k' \rightfun] < \varsigma_{K}.
\end{equation*}

As noted above this analysis is incorrect.
The problem is that \(\gamma\) is \emph{never} an isolated critical
point because of gauge symmetry. The complete hyperbolic structure
of \(K\) corresponds to the entire gauge class \([\rho]\) of a 
representation \(\rho\). The gauge transformations of
\cref{def:gauge transformation} show that a given representation
\(\rho\) corresponds to a manifold of complex dimension \(3\) lying in the
space \(\mathbb{C}^{E}\) of segment parameters.
One degree of freedom is trivial, corresponding to simultaneously scaling 
all segment parameters.
It is fixed by restricting to normalized shadow colorings. Generically the other two degrees of freedom
are nontrivial, so we expect \(\gamma\) to lie in a critical manifold of complex dimension \(2\).
These nontrivial degrees of freedom can be understood as coming from the two material 
vertices of the octahedral decomposition, as in \cite[Example 4.6]{Kim2018}.
Generically the integration contour will intersect this manifold in \(2\) real dimensions, recovering an overall power of \(\nr\).

This heuristic dimension counting can be made more precise.
To establish it rigorously we would need to show  the set of critical
points is a manifold \(B\)
(i.e.\ that we have a \defemph{Morse-Bott} critical manifold) and
that the transverse Hessian is nondegenerate;
this nondegeneracy is related to the assertion below that \(\tau \ne 0\).
Furthermore the integration contour must intersect transversely with \(B\).
We expect the explicit coordinates worked out in
\cite{McphailSnyder2024octahedralcoordinateswirtingerpresentation}
will be useful in establishing these properties.
Alternatively, the intersection behavior could be more complicated; we discuss this below.

Experimental evidence \cite[Section 6.2, Remark
6.5]{McPhailSnyderVolume} provides one more
hypothesis for our conjecture.
Numerical computations suggest that for the \(K = 4_{1}\) the figure-eight knot
and \(\rho\) any nontrivial, boundary-parabolic representation we have
\begin{equation*}
  \qinv{4_{1}, \rho, \mu}
  =
  0
  \text{ unless }
  2 \mu \equiv -1 \pmod{\mathbb{Z}}
\end{equation*}
There may be a representation-theoretic explanation for this
vanishing; see \cite[Conjecture 3]{McPhailSnyderVolume}.
The same computations suggest
\begin{equation*}
  \qinv{4_{1}, \rho, \mu + \nr /2}
  =
  -
  \qinv{4_{1}, \rho, \mu}
\end{equation*}
which means different choices of \(\mu\) correspond to different lifts
of \(\rho_{\hyp}\), hence to the sign of \(e^{ \varsigma_{K}}\).

\begin{conjecture}
  \label{our volume conjecture}
  Let \(K\) be a hyperbolic knot and \(\rho\) a decorated, boundary-parabolic
  representation.%
  \note{
    Our assumption on \(\mu\) implies that \(\rho\) must have
    meridian eigenvalue \(-1\) but this could be dropped at the cost
    of some extra signs in the equation below.
  }
  Let \(\varsigma_{K}\) be the complex Chern-Simons
  invariant of the lift \(\rho_{\hyp}\) of the complete hyperbolic
  structure of \(K\) with meridian eigenvalue \(1\), normalized so
  that \(\Re \varsigma_{K} = \vol(K)/2\pi\).
  (See \cite[Example 4.6]{McPhailSnyderVolume} for more details on
  the normalization of \(\varsigma_{K}\).)

  Assume that \(\rho\) is geometrically nondegenerate in
  the sense that \(\rho\) is not pinched for some open diagram \(D\) of \(K\).
  Then there is a nonzero constant \(\tau\) so that
  \begin{equation*}
    \qinv{K, \rho, -1/2}
    =
    \nr^{1/2}
    e^{\nr \varsigma_{K} }
    [\tau + \bigoh(\nr^{-1}) ]
    \qedhere
  \end{equation*}
\end{conjecture}

Observe that the polynomial term is \(\nr^{1/2}\), not the expected  \(\nr^{3/2}\) as in \cite{zbMATH06684922}.
It would be good to better understand this discrepancy.
One possibility is that the intersection of the critical manifold and the contour cannot be made transverse.
\textcite{zbMATH05624995} made similar conjectures and in the process observed that the expected power of \(\nr\) should be related to dimension of gauge orbits in the moduli space of hyperbolic structures. 
However their analysis predicts \(\nr^{3/2}\) here, not \(\nr^{1/2}\).

\subsection{Barriers to moving the contour}%
\label{sec:Barriers to Step 2}

\begin{figure}
  \centering
  \begin{tikzpicture}[line width=1]
    \draw (-2,0) to (2,0);
    \draw (0,-0.5) to (0,1);
    \draw[line width = 2, color=slred] (-2, 0) node[below] {\(X\)} to (0,0);
    \draw[line width = 2, color=slred] (1, 0) to (2, 0);
    \draw[->] (0,0) to (1.5, 1) node[above] {\(\zeta\)};
    \draw[color=slblue] (1.5-1,1) to (2.5, 1) node[right] {\(\zeta
    + [-1,1]\)};
    \draw[color=slblue] (.75-1,.5) to (1.75,.5) node[right]
    {\(\zeta/2 + [-1,1]\)};
  \end{tikzpicture}
  \caption{
    The arguments of the dilogarithms in the state integrals \(
    \stint{D, \mu, \beta}{k} \) are of the form \(\zeta/\nr + t_{i} -
    t_{j}\) with \(t_{i}, t_{j} \in [0,1]\).
    As \(\nr \to \infty\) the segments over which they vary approach
    the singular region \(X = (-\infty, 0] \cup [1, \infty)\).
  }
  \label{fig:approach singularities}
\end{figure}
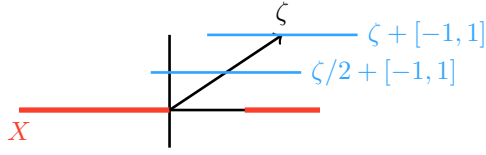

Establishing the volume asymptotics of \cref{our volume conjecture}
requires replacing our original state integrals
\begin{equation*}
  \stint{D, \mu, \beta}{k} \defeq
  \int_{[0,1]^{E}}
  \exp \nr \left( \actq{D, \mu}{t + \beta/\nr} - \tu k \cdot t \right)
  \dif t
\end{equation*}
with shifted ones
\[
  \stintadj{D}{k} \defeq
  \int_{\Gamma}
  \exp \nr \left( \act{D}{t} - \tu k \cdot t \right)
  \dif t.
\]
This is \cref{step 2} for \(\qinv{}\).
The idea is to use Stokes' Theorem to shift the initial contour
\([0,1]^{E}\) to some new contour \(\Gamma\) in saddle position for
the correct critical point.
The obvious problem is that Stokes' Theorem can move only the
interior of the contour, not the boundary.%
\note{%
  If the integrand were periodic the boundary terms would cancel but
  this only holds on a lattice in \(\mathbb{C}^{E}\) as in
  \cref{thm:action lemmas}\ref{thm:action lemma:lattice periodic}.
}
This is a serious issue.
As shown in \cref{fig:approach singularities} the arguments of the
quantum dilogarithms in the action in the original state integral
approach the real axis as \(\nr \to \infty\).
Unfortunately the real axis is precisely where the problems are: as
\(\nr \to \infty\) the singularities of \(e^{\pm \pfl{}}\) become
dense in \(X\).
Near this region \(\actq{D, \mu}{}\)  can fail to uniformly converge
to \(\act{D}{}\), which is related to the fact that the branch points
of the dilogarithm lie there.
While the interior of the contour can be moved away the remaining
boundary will still approach the singular region.
Understanding why the remaining boundary terms cancel or do not
contribute is a hard analytic problem.
We expect that solving it requires new ideas.

There may be subtleties here related to the vanishing observed above and in
\cite[Conjecture 3]{McPhailSnyderVolume}.
When \(2 \mu \not \equiv -1\) the relevant state integrals do not
appear radically different, but experimental evidence suggests
\cref{our volume conjecture} is false.
The expected failure of asymptotic growth indicates that there can be nontrivial
cancellations in the sum \eqref{eq:state integral sum} and/or that
there are other barriers to moving the contour. Examining these
issues may shed light on why it sometimes \emph{is} possible to move
the contour and prove \cref{our volume conjecture}.

\printbibliography

\end{document}

%% file: macros-integrals.tex
\newcommand{\nr}{N}
\newcommand{\set}[1]{\left\{ \def\given{\ \middle| \ }  #1 \right\}  }
\newcommand{\defeq}{\mathrel{:=}}

\NewDocumentCommand{\nset}{m}{\left[#1\right]}

\renewcommand{\Re}{\operatorname{Re}}
\renewcommand{\Im}{\operatorname{Im}}

\NewDocumentCommand{\dif}{}{\,\mathrm{d}}

\NewDocumentCommand{\bigoh}{}{\mathrm{O}}

\ExplSyntaxOn
\DeclareExpandableDocumentCommand{\IfEmptyTF}{mmm}
 {
    \tl_if_empty:nTF {#1} {#2} {#3}
 }
\ExplSyntaxOff

\def\leftfun#1{\mathopen{}\left#1}
\def\rightfun#1{\right#1}

\NewDocumentCommand{\ParenIfNonempty}{m}{%
  \IfEmptyTF{#1}{%
  }{%
    \leftfun(#1\rightfun)%
  }
}

\NewDocumentCommand{\ii}{}{\mathsf{i}}
\NewDocumentCommand{\tu}{}{2 \pi \ii}

\NewDocumentCommand{\evec}{m}{\mathbf{e}_{#1}}

\NewDocumentCommand{\slg}{}{\operatorname{SL}_2(\mathbb{C})}

\NewDocumentCommand{\pjsp}{O{\mathbb{C}}}{\operatorname{#1 P}^{1}}
\DeclareMathOperator{\hyp}{hyp}
\DeclareMathOperator{\vol}{Vol}
\NewDocumentCommand{\qn}{}{\triangleleft}
\NewDocumentCommand{\qninv}{}{\mathbin{\overline{\triangleleft}}}

\NewDocumentCommand{\mer}{}{\mathbf{m}}
\NewDocumentCommand{\comp}{m}{S^3 \setminus #1}

\NewDocumentCommand{\qfunc}{O{\nr} m}{\mathcal{Z}_{#1}^{\psi}\ParenIfNonempty{#2}}
\NewDocumentCommand{\qinv}{O{\nr} m}{\mathrm{Z}_{#1}^{\psi}\ParenIfNonempty{#2}}

\NewDocumentCommand{\actq}{O{\nr} m m}{%
  \mathcal{S}_{#1}\leftfun(#2
  \IfEmptyTF{#3}{%
    \rightfun)
  }{%
    \middle|#3\rightfun)%
  }
}

\NewDocumentCommand{\act}{ m m}{%
  \mathcal{S}\leftfun(#1
  \IfEmptyTF{#2}{%
    \rightfun)
  }{%
    \middle|#2\rightfun)%
  }
}

\NewDocumentCommand{\actadj}{ m m}{%
  \widetilde{\mathcal{S}}\leftfun(#1
  \IfEmptyTF{#2}{%
    \rightfun)
  }{%
    \middle|#2\rightfun)%
  }
}

\NewDocumentCommand{\stint}{O{\nr} m m}{%
  J_{#1}\leftfun(#2
  \IfEmptyTF{#3}{%
    \rightfun)
  }{%
    \middle|#3\rightfun)%
  }
}

\NewDocumentCommand{\stintadj}{O{\nr} m m}{%
  \widetilde{J}_{#1}\leftfun(#2
  \IfEmptyTF{#3}{%
    \rightfun)
  }{%
    \middle|#3\rightfun)%
  }
}

\NewDocumentCommand{\crossingfunction}{O{\nr} m m m}{
  \Omega_{#1}^{#2}\leftfun(#3
  \IfEmptyTF{#4}{%
    \rightfun)
  }{%
    \middle|#4\rightfun)%
  }
}

\NewDocumentCommand{\crossingfunctionclass}{m m}{
  \Omega^{#1}
  \IfEmptyTF{#2}{%
  }{%
    \leftfun(#2\rightfun)%
  }
}

\NewDocumentCommand{\pfl}{O{\nr} m}{\phi_{#1}\ParenIfNonempty{#2}}

\NewDocumentCommand{\qlfl}{O{\nr} m m}{%
  \lambda_{#1}\leftfun(#2
  \IfEmptyTF{#3}{%
    \rightfun)
  }{%
    \middle|#3\rightfun)%
  }
}

\NewDocumentCommand{\dll}{m }{\operatorname{L}\ParenIfNonempty{#1}}

\NewDocumentCommand{\lN}{}{\mathrm{N}}
\NewDocumentCommand{\lW}{}{\mathrm{W}}
\NewDocumentCommand{\lS}{}{\mathrm{S}}
\NewDocumentCommand{\lE}{}{\mathrm{E}}

\newcommand{\br}{to [out=00,in=180]}